\documentclass[11pt,reqno]{amsart}

\usepackage{amsmath,amstext,amssymb,amscd}
\usepackage{verbatim}
\usepackage{enumerate}
\usepackage{mathrsfs}
\usepackage[dvipsnames,usenames]{xcolor}
\usepackage{calc}

\usepackage{mathtools}

\usepackage{amsthm}

\newtheorem{theorem}{Theorem}[section]

\newtheorem{proposition}[theorem]{Proposition}

\theoremstyle{definition}

\theoremstyle{remark}
\newtheorem{remark}[theorem]{Remark}

\numberwithin{equation}{section}
\begin{document}

\def\bar{\overline}

\author[C. Cao]{Chongsheng Cao}
\address{Department of Mathematics \& Statistics \\Florida International University\\
Miami, Florida 33199, USA} \email{caoc@fiu.edu}

\author[Y. Guo]{Yanqiu Guo}
\address{Department of Mathematics \& Statistics \\Florida International University\\
Miami, Florida 33199, USA} \email{yanguo@fiu.edu}

\author[E. S. Titi]{Edriss S. Titi}
\address{Department of
Mathematics \\ Texas A\&M University\\
College Station, TX 77843, USA;
Department of Applied Mathematics and Theoretical Physics, University of Cambridge, Cambridge CB3 0WA, UK; 
\textbf{AND} Department of Computer Science and Applied Mathematics \\ Weizmann Institute of Science \\ Rehovot 7610001, Israel}
 \email{titi@math.tamu.edu,  edriss.titi@maths.cam.ac.uk}


\title[A rapidly rotating convection model without thermal diffusion]
{Analysis of a three-dimensional  rapidly rotating convection model without thermal diffusion}
\date{July 21, 2025}
\maketitle

\begin{abstract}
We study a three-dimensional rapidly rotating convection model featuring tall columnar structures, in the absence of thermal diffusion. We establish the global existence and uniqueness of weak solutions, as well as the Hadamard well-posedness of global strong solutions to this model. The lack of thermal diffusion introduces significant challenges in the analysis. To overcome these challenges, we first investigate the regularized model with thermal diffusion and establish delicate estimates that are independent of the thermal diffusion coefficient, and consequently justify the  vanishing diffusivity limit. This work serves as a continuation of our previous paper \cite{CGT-3}.
\end{abstract}

\section{Introduction}
The classical Boussinesq approximation, or Rayleigh–Bénard convection, consists of a system of equations coupling the three-dimensional Navier–Stokes equations with a heat advection–diffusion equation. By including the Coriolis force, one obtains the three-dimensional rotating Boussinesq equations:
\begin{align}
& \frac{\partial \vec u}{\partial t}  +  \frac{1}{Ro} \vec k \times \vec u +    \vec u \cdot \nabla \vec u +  \nabla p     
=    \frac{1}{Re}    \Delta \vec u  +  \Gamma    \theta    \vec k,    \label{bou1}    \\
& \frac{\partial \theta}{\partial t} +   \vec u \cdot \nabla \theta =  \frac{1}{Pe}    \Delta \theta,    \label{bou2}
\end{align}
subject to the incompressibility condition $\nabla \cdot \vec u = 0$, where $\vec u = (u, v, w)$ is the velocity vector field in three dimensions. The scalar functions $\theta$ and $p$ represent the temperature and pressure, respectively. Let $\vec k$ denote the unit vector in the vertical direction, aligned with the axis of rotation. The term $\frac{1}{Ro} \vec k \times \vec u$ represents the Coriolis force. The system involves several dimensionless parameters: the Rossby number $Ro$, the Reynolds number $Re$, the P\'eclet number $Pe$, and the buoyancy number $\Gamma$  (see, e.g., \cite {MajdaAtmosphereOcean}). Due primarily to the difficulties arising from the three-dimensional Navier–Stokes equations, the global regularity of the 3D Boussinesq system \eqref{bou1}–\eqref{bou2} remains an open problem.

It is well known that the presence of fast rotation in hydrodynamic models can prolong the lifespan of solutions.
This phenomenon arises because fast rotation introduces an averaging, dispersive-like, effect that help stabilize the flow.  For instance, Babin et al.~\cite{BMN97,BMN99} showed that, for the 3D rotating Euler equations, the lifespan of regular solutions extends to infinity as the rotation rate tends to infinity. Moreover, the 3D rotating Navier--Stokes equations admit global regular solutions if the rotation rate is sufficiently large, depending on the initial data. In a related result, Ghoul et al. \cite{Ghoul} showed that, in the space of analytic functions, the inviscid primitive equations for oceanic and atmospheric dynamics are locally well-posed, and the lifespan of the solution tends to infinity as the rotation rate increases to infinity. These findings motivate the study of Rayleigh--Bénard convection under the influence of fast rotation, as rapid rotation may extend the existence time of solutions.

Coherent structures in convection under moderate rotation are observed to be exclusively cyclonic. However, under rapid rotation, experiments have revealed a transition to approximately equal populations of cyclonic and anticyclonic structures. For a laboratory experimental study, see Vorobieff and Ecke~\cite{VE}. To numerically investigate these cyclonic and anticyclonic coherent structures in Rayleigh--Bénard convection under the influence of rapid rotation, Sprague et al.~\cite{SJKW} introduced a reduced system for rotationally constrained convection, valid in the asymptotic limit of thin columnar structures and strong rotation (see equations~(\ref{m-1})--(\ref{m-4}) below). The derivation of the model (\ref{m-1})–(\ref{m-4}) for rapidly rotating convection was also presented by Julien and Knobloch in \cite{JK}. In fact, system (\ref{m-1})–(\ref{m-4}) was derived from the Boussinesq equations (\ref{bou1})–(\ref{bou2}) using asymptotic theory. The dependent variables in the Boussinesq equations were expanded in terms of the small parameter, the Rossby number $Ro$. It was also assumed that $\frac{L}{H} \sim Ro$, where $\frac{L}{H}$ is the aspect ratio of the horizontal scale $L$ to the vertical scale $H$ in the original Boussinesq system (\ref{bou1})–(\ref{bou2}). In addition, a slow time $T \sim Ro^2\, t$ was introduced. By considering fast rotation and tall columnar structures, the limit $Ro\rightarrow 0$ was taken, leading to the following closed system derived in \cite{JK, SJKW}:
\begin{align}    
&\frac{\partial w}{\partial t}+ \mathbf u \cdot \nabla_h w + \frac{\partial \psi}{\partial z}= \frac{Ra}{Pr}\theta' + \Delta_h w,   \label{m-1} \\
&\frac{\partial \omega}{\partial t} + \mathbf u \cdot \nabla_h \omega - \frac{\partial w}{\partial z}= \Delta_h \omega,   \label{m-2}\\
&\frac{\partial \theta'}{\partial t}  + \mathbf u \cdot \nabla_h \theta' +  w \frac{\partial \bar{\theta}}{\partial z}= \frac{1}{Pr} \Delta_h \theta',   \label{m-3}\\
&\frac{\partial \bar{\theta}}{\partial T} + \frac{\partial}{\partial z} (\bar{\theta' w}) = \frac{1}{Pr} \frac{\partial^2 \bar \theta}{\partial z^2},      \label{m-4} 
\end{align}
subject to the horizontal incompressibility condition $\nabla_h \cdot \mathbf u=0$.
For simplicity of analysis, we consider system (\ref{m-1})–(\ref{m-4}) on a three-dimensional periodic domain $\Omega=[0,2\pi]^3$. 
We denote by $\mathbf u=(u,v)$ the horizontal component of the three dimensional velocity field $(u,v,w)$. Here, $\nabla_h=(\frac{\partial}{\partial x}, \frac{\partial}{\partial y})$ is the horizontal gradient, and $\Delta_h=\frac{\partial^2}{\partial x^2}+ \frac{\partial^2}{\partial y^2}$ is the horizontal Laplacian. The scalar function $\omega=  \nabla_h \times \mathbf u=\partial_x v - \partial_y u$ denotes the vertical component of the vorticity. Moreover, $\psi=\Delta_h ^{-1}\omega$ is the stream function for the horizontal flow, which is determined uniquely by assuming that its horizontal spatial average/mean is zero. The temperature fluctuation is given by
$\theta'=\theta - \bar{\theta}$, where 
$\bar{\theta}(z)$ is the horizontal mean temperature. Throughout the paper, the horizontal mean of any function $f$ is defined by
$\bar{f}(z) = \frac{1}{4 \pi^2}\int_{[0,2\pi]^2} f(x,y,z) dx dy$.  In the system above, $Ra$ is the Rayleigh number and $Pr$ is the Prandtl number.

Let us highlight some elements used in the derivation of model~(\ref{m-1})--(\ref{m-4}) from the papers~\cite{JK, SJKW}. Due to the rapid rotation, the pressure gradient force is balanced by the Coriolis effect---an approximate state known as \emph{geostrophic balance}. Furthermore, the Taylor--Proudman constraint suggests that rapidly rotating convection occurs in tall, columnar structures. We emphasize that the variable~$z$ in system~(\ref{m-1})--(\ref{m-4}) is associated with a large vertical scale, whereas~$x$ and~$y$ correspond to small horizontal scales. As a result, vertical viscosity is negligible, and the system~(\ref{m-1})--(\ref{m-4}) includes only horizontal viscosity.

The global well-posedness of system~(\ref{m-1})--(\ref{m-4}) remains an open problem. The main difficulty in analyzing~(\ref{m-1})--(\ref{m-4}) lies in the fact that the physical domain is three-dimensional, while the regularizing viscosity acts only on the horizontal variables. Moreover, the equations contain the terms $\frac{\partial \psi}{\partial z}$ and $\frac{\partial w}{\partial z}$, which involve vertical derivatives and pose challenges for establishing global regularity. In our previous work~\cite{CGT-2}, we introduced a weak dissipation term to regularize system~(\ref{m-1})--(\ref{m-4}), and established the global well-posedness of strong solutions for the regularized system. More precisely, in~\cite{CGT-2}, we added a very weak numerical dissipation term, $\frac{\partial^2 \psi}{\partial z^2}$, to equation~(\ref{m-2}), which stabilizes the fluid in the vertical direction and helps control the problematic terms $\frac{\partial \psi}{\partial z}$ and $\frac{\partial w}{\partial z}$. We remark that the Taylor--Proudman theory suggests that, in rapidly rotating fluids, the Coriolis force dominates and suppresses vertical variations in the flow, so the motion can be regarded as effectively $2 + \varepsilon$ dimensional. Therefore, it is of interest to study the global well-posedness of reduced models for rapidly rotating convection, such as system~(\ref{m-1})--(\ref{m-4}). Additionally, in plasma physics, the Hasegawa–Mima equation \cite{HM}, which models plasma turbulence in tokamaks, shares a similar structure with system~(\ref{m-1})--(\ref{m-4}). 
In \cite{CGT-1}, we studied global strong solutions for the 3D Hasegawa-Mima model with partial dissipation. Moreover, the global well-posedness of an inviscid 3D pseudo-Hasegawa–Mima model was established by Cao et al. \cite{CFT}. However, the global well-posedness of the original 3D  Hasegawa–Mima system remains an open problem.

In a previous paper \cite{CGT-3}, we studied the fast rotating convection system (\ref{m-1})–(\ref{m-4}) in the infinite Prandtl number limit (see system (\ref{inP-1})–(\ref{inP-4}) below), and established the global well-posedness of both weak and strong solutions. The current manuscript continues this line of investigation. At large Prandtl numbers, viscosity dominates thermal diffusivity. Motivated by this, we examine whether the system remains globally well-posed when the thermal diffusion term is entirely removed after taking the infinite Prandtl number limit. This system is challenging, as it eliminates the regularizing effect provided by thermal diffusion. In this work, we resolve this issue by proving the global well-posedness of system (\ref{inPn-1})–(\ref{inPn-4}) presented below.

\vspace{0.1 in}

\section{Main results}

In this section, we present our main results. Before doing so, we introduce the model equations (\ref{inPn-1})–(\ref{inPn-4}) that will be the focus of this manuscript. The main results include the global existence and uniqueness of weak solutions, the global well-posedness of strong solutions, and the justification of the vanishing diffusivity limit.

\subsection{The Model}

We are interested in the behavior of system (\ref{m-1})--(\ref{m-4}) in the regime where viscosity dominates thermal diffusion. To explore this, we consider the formal limit as the Prandtl number \( Pr \to \infty \), since the Prandtl number represents the ratio of kinematic viscosity to thermal diffusivity. In this regime, the appropriate time scale is the horizontal thermal diffusion time. Following \cite{JK, SJKW}, we apply the scaling transformation:
\[
t \rightarrow Pr\, t, \quad \mathbf{u} \rightarrow \frac{1}{Pr} \mathbf{u}, \quad w \rightarrow \frac{1}{Pr} w,
\]
to system (\ref{m-1})--(\ref{m-4}).  This transformation stretches the time scale to match the slower thermal diffusion. For analytical simplicity, we omit the slow time variable \( T \) by removing the term \( \bar{\theta}_T \). Taking the limit \( Pr \to \infty \), in the re-scaled equations yields the following formal limit system:
\begin{align}    
&\psi_z = Ra\,\theta' + \Delta_h w,   \label{inP-1} \\\
& - w_z = \Delta_h \omega,   \label{inP-2} \\\
&\theta'_t + \mathbf{u} \cdot \nabla_h \theta' + w \bar{\theta}_z = \Delta_h \theta',   \label{inP-3} \\\
&(\bar{\theta' w})_z = \bar{\theta}_{zz},     \label{inP-4}
\end{align}
with \( \nabla_h \cdot \mathbf{u} = 0 \). In system (\ref{inP-1})--(\ref{inP-4}), the velocity field adjusts instantaneously to thermal fluctuations. In our previous work \cite{CGT-3}, we established the global well-posedness of both weak and strong solutions to  system (\ref{inP-1})--(\ref{inP-4}) using the Galerkin method and energy estimates. A key step in our analysis was using the explicit solutions of the linear equations (\ref{inP-1})–(\ref{inP-2}), which allowed us to transfer the smoothing effect from the horizontal direction to the vertical direction.

In this paper, we study a more challenging problem inspired by system (\ref{inP-1})--(\ref{inP-4}). 
Since we are interested in the regime where viscosity dominates thermal diffusivity, 
a natural question arises: \emph{If the thermal diffusion term \( \Delta_h \theta' \) is removed from equation (\ref{inP-3}), 
does the system remain globally well-posed}? More precisely, we investigate the system:
\begin{align}    
& \psi_z = \theta' + \Delta_h w,   \label{inPn-1} \\\
& - w_z = \Delta_h \omega,   \label{inPn-2} \\\
& \theta'_t + \mathbf{u} \cdot \nabla_h \theta' + w \bar{\theta}_z = 0,   \label{inPn-3} \\\
& (\bar{\theta' w})_z = \bar{\theta}_{zz},     \label{inPn-4}
\end{align}
with \( \nabla_h \cdot \mathbf{u} = 0 \), on the periodic domain \( \Omega = [0, 2\pi]^3 \). 
For simplicity, we have set \( Ra = 1 \). We emphasize that the absence of the regularizing thermal diffusion term in system (\ref{inPn-1})--(\ref{inPn-4}) 
makes the global well-posedness problem significantly more difficult. 
Nevertheless, in this paper, we succeed in proving the global existence and uniqueness of weak solutions, as well as the Hadamard well-posedness of global strong solutions to system (\ref{inPn-1})--(\ref{inPn-4}). 
Moreover, we justify the vanishing diffusivity limit, demonstrating that model (\ref{inPn-1})--(\ref{inPn-4}) 
serves as a valid approximation of model (\ref{inP-1})--(\ref{inP-4}) when the thermal diffusivity is sufficiently small. In our analysis, we use the explicit solutions of the linear equations (\ref{inPn-1})–(\ref{inPn-2}) and apply the theory of $L^p$ Fourier multipliers. To establish the uniqueness of weak solutions, we employ a specially chosen low-regularity negative Sobolev space.

\vspace{0.1 in}

\subsection{Notations}

We consider the three-dimensional periodic domain \( \Omega = [0, 2\pi]^3 \).  
Let \( A = -\Delta_h \), where $\Delta_h=\frac{\partial^2}{\partial x^2}+ \frac{\partial^2}{\partial y^2}$ denotes the horizontal Laplacian acting on the appropriate space of functions with zero horizontal spatial mean. For \( s \geq 0 \), we define the space
\[
H^s_h(\Omega) = \left\{ f \in L^2(\Omega) : A^{s/2} f \in L^2(\Omega) \;\; \text{and} \;\; \bar{f} = 0 \right\},
\]
with the norm \( \|f\|_{H^s_h(\Omega)} = \|A^{s/2} f\|_{L^2(\Omega)} \).  

Here, the operator \( A^{s/2} \) is defined via Fourier series as
\[
A^{s/2} f = \sum_{\mathbf k = (k_1, k_2, k_3) \in \mathbb{Z}^3} (k_1^2 + k_2^2)^{s/2} \hat{f}_{\mathbf k} e^{i(k_1 x + k_2 y + k_3 z)},
\]
where \( \hat{f}_{\mathbf k} \) are the Fourier coefficients of \( f \), and  \( \bar{f}(z) = \frac{1}{4\pi^2} \int_{[0, 2\pi]^2} f(x, y, z) \, dx dy \) denotes the horizontal average of \( f \).

The dual space of \( H^s_h(\Omega) \) is denoted by \( (H^s_h(\Omega))' \).

We denote the \( L^p(\Omega) \) norm by \( \|\cdot\|_p = \|\cdot\|_{L^p(\Omega)} \).

We also define the space
\[
\dot{H}^1(0, 2\pi) = \left\{ g \in H^1(0, 2\pi) : \int_0^{2\pi} g(z) \, dz = 0 \right\},
\]
and denote its dual by \( H^{-1}(0, 2\pi) = (\dot{H}^1(0, 2\pi))' \).

\vspace{0.1 in}

\subsection{Statements of Main Results}

Our first result establishes the global existence and uniqueness of weak solutions to system (\ref{inPn-1})--(\ref{inPn-4}), assuming the initial data \( \theta'_0 \) belongs to the space \( L^6(\Omega) \).

\begin{theorem}[\textbf{Global existence and uniqueness of weak solutions}]  \label{wellp}
Assume the initial data \( \theta'_0 \in L^6(\Omega) \) satisfies \( \bar{\theta'_0} = 0 \), and let $T>0$ be arbitrarily large. 
Then system (\ref{inPn-1})--(\ref{inPn-4}) admits a unique global weak solution \( (\theta', \bar{\theta}, \mathbf{u}, w) \) enjoying the following regularity properties:
\[
\theta' \in  L^{\infty}(0,T; L^6(\Omega)) \cap C_w(0,T; L^2(\Omega)), 
\]
\[
\bar{\theta} \in L^2(0,T; \dot{H}^1(0,2\pi)),  \quad    \Delta_h \mathbf{u}, \Delta_h w \in L^{\infty}(0,T; L^6(\Omega)).
\]
Moreover, the equations hold in the following spaces:
\begin{align}
& \psi_z = \theta' + \Delta_h w, && \text{in } L^{\infty}(0,T; L^6(\Omega)), \label{iR-1} \\
& -w_z = \Delta_h \omega, && \text{in } L^{\infty}(0,T; (H^1_h(\Omega))'), \label{iR-2} \\
& \theta'_t + \mathbf{u} \cdot \nabla_h \theta' + w \bar{\theta}_z = 0, && \text{in } L^2(0,T; (H^1_h(\Omega))'), \label{iR-3} \\
& (\overline{\theta' w})_z = \bar{\theta}_{zz}, && \text{in } L^2(0,T; H^{-1}(0,2\pi)), \label{iR-4}
\end{align}
with \( \nabla_h \cdot \mathbf{u} = 0 \) and \( \theta'(0) = \theta'_0 \).
\end{theorem}

\vspace{0.1 in}

To prove the existence of global weak solutions, we regularize the system by adding back a horizontal thermal diffusion term $\varepsilon^2 \Delta_h \theta'$ to equation (\ref{iR-3}), with  $\varepsilon > 0$ arbitrarily small to be eventually sent to zero. We then perform delicate energy estimates in the $L^2(\Omega)$, $L^3(\Omega)$, and $L^6(\Omega)$ norms to the regularized system, deriving uniform bounds that are independent of $\varepsilon$. The explicit linear relations between $\theta'$ and the velocity in equations (\ref{iR-1})--(\ref{iR-2}), along with the $L^p$ Fourier multiplier theory, play a crucial role in these estimates. The uniqueness of weak solutions is established using a norm in a dual space, following the approach of Yudovich \cite{Yu} and Larios et al. \cite{LLT}. However, due to the weak norm employed in the uniqueness proof, this argument cannot be directly adapted to show continuous dependence on the initial data.

To establish continuous dependence, we require more regular initial data. Theorem \ref{wellp2} below proves the global existence, uniqueness, and continuous dependence on initial data for strong solutions. These strong solutions will be essential in justifying the vanishing diffusivity limit in Theorem \ref{thm-conv}.

\begin{theorem}[\textbf{Hadamard well-posedness of global strong solutions}] \label{wellp2}
Assume the initial data $\theta'_0 \in L^6(\Omega)$ satisfies $\nabla_h \theta'_0 \in L^3(\Omega)$ and $\bar{\theta_0'} = 0$. Let $T>0$ be arbitrarily large.  
Then  system \eqref{inPn-1}--\eqref{inPn-4} admits a unique global strong solution $(\theta', \bar{\theta}, \mathbf{u}, w)$ satisfying the following regularity properties:
\[
\theta' \in  L^{\infty}(0,T; L^6(\Omega)) \cap  C(0,T; L^2(\Omega))   , \quad \nabla_h \theta' \in L^{\infty}(0,T; L^3(\Omega)),
\]
\[
\bar{\theta} \in L^{\infty}(0,T; \dot{H}^1(0,2\pi)), \quad (\mathbf{u}, w) \in L^{\infty}(0,T; W^{1,3}(\Omega)),
\]
and the equations hold in the following function spaces:
\begin{align}
& \psi_z = \theta' + \Delta_h w, && \text{in } L^{\infty}(0,T; H^1_h(\Omega)), \label{iRR-1} \\
& -w_z = \Delta_h \omega, && \text{in } L^{\infty}(0,T; L^3(\Omega)), \label{iRR-2} \\
& \theta'_t + \mathbf{u} \cdot \nabla_h \theta' + w \bar{\theta}_z = 0, && \text{in } L^2(\Omega \times (0,T)), \label{iRR-3} \\
& (\overline{\theta' w})_z = \bar{\theta}_{zz}, && \text{in } L^2(0,T; H^{-1}(0,2\pi)), \label{iRR-4}
\end{align}
with $\nabla_h \cdot \mathbf{u} = 0$ and $\theta'(0) = \theta'_0$. Moreover, the solution depends continuously on the initial data. Specifically, for a sequence of initial data $\theta'_{0,n} \to \theta'_0$ in $L^2(\Omega)$, where $\theta'_{0,n}, \theta'_0 \in L^6(\Omega)$, $\nabla_h \theta'_{0,n}, \nabla_h \theta'_0 \in L^3(\Omega)$, and $\bar{\theta'_{0,n}} = \bar{\theta'_0} = 0$, for $n=1,2,\dots$, the corresponding strong solutions satisfy:
\[
\theta'_n \to \theta' \text{ in } L^{\infty}(0,T; L^2(\Omega)), \quad \bar{\theta}_n \to \bar{\theta} \text{ in } L^{\infty}(0,T; \dot{H}^1(0,2\pi)),
\]
\[
(\mathbf{u}_n, w_n) \to (\mathbf{u}, w) \text{ in } L^{\infty}(0,T; H^2_h(\Omega)),
\]
with $\theta'_n(0) = \theta'_{0,n}$ and $\theta'(0) = \theta'_0$.
\end{theorem}

\vspace{0.1 in}

Our final result concerns the justification of the vanishing diffusivity limit. In our previous work \cite{CGT-3}, we established the global well-posedness of weak solutions  
$(\theta'_{\varepsilon}, \bar{\theta}_{\varepsilon}, \mathbf{u}_{\varepsilon}, w_{\varepsilon})$
to the system:
\begin{align}
& \partial_z \psi_{\varepsilon} = \theta'_{\varepsilon} + \Delta_h w_{\varepsilon}, && \text{in } L^2(0,T; H^1_h(\Omega)), \label{fR-1} \\
& -\partial_z w_{\varepsilon} = \Delta_h \omega_{\varepsilon}, && \text{in } L^2(\Omega \times (0,T)), \label{fR-2} \\
& \partial_t \theta'_{\varepsilon} + \mathbf{u}_{\varepsilon} \cdot \nabla_h \theta'_{\varepsilon} + w_{\varepsilon} \partial_z \bar{\theta}_{\varepsilon} = \varepsilon^2 \Delta_h \theta'_{\varepsilon}, && \text{in } L^2(0,T; (H^1_h(\Omega))'), \label{fR-3} \\
& \partial_z(\overline{\theta'_{\varepsilon} w_{\varepsilon}}) = \partial_{zz} \bar{\theta}_{\varepsilon}, && \text{in } L^2(0,T; H^{-1}(0,2\pi)), \label{fR-4}
\end{align}
with $\nabla_h \cdot \mathbf{u}_{\varepsilon} = 0$, and initial data $\theta'_{\varepsilon}(0) \in L^2(\Omega)$ and $\bar{\theta'_{\varepsilon}(0) } =0$.

The following theorem states that the weak solution of system \eqref{fR-1}--\eqref{fR-4} converges in the $L^2$ norm to the strong solution of system \eqref{iRR-1}--\eqref{iRR-4} as $\varepsilon \to 0$, provided the initial data converge correspondingly in the $L^2$ norm. 

\begin{theorem}[\textbf{Vanishing diffusivity limit}] \label{thm-conv}
Let $(\theta'_{\varepsilon}, \bar{\theta}_{\varepsilon}, \mathbf{u}_{\varepsilon}, w_{\varepsilon})$ be the weak solution of system \eqref{fR-1}--\eqref{fR-4} 
with $\theta'_{\varepsilon}(0) \in L^2(\Omega)$ and $\bar{\theta'_{\varepsilon}(0) } =0$. Let $(\theta', \bar{\theta}, \mathbf{u}, w)$ be the strong solution of system \eqref{iRR-1}--\eqref{iRR-4} such that $\theta'(0) = \theta'_0 \in L^6(\Omega)$ with $\nabla_h \theta'_0 \in L^3(\Omega)$ and $\bar{\theta'_0} =0$. Let $T > 0$.  
Then, as $\varepsilon \to 0$, the solution $(\theta'_{\varepsilon}, \bar{\theta}_{\varepsilon}, \mathbf{u}_{\varepsilon}, w_{\varepsilon})$ converges strongly to $(\theta', \bar{\theta}, \mathbf{u}, w)$ in the norm of the space $L^{\infty}(0,T; L^2(\Omega) \times \dot{H}^1(0,2\pi) \times (H^2_h(\Omega))^3)$, provided that $\|\theta'_{\varepsilon}(0) - \theta'_0\|_2 \to 0$. In particular, if $\|\theta'_{\varepsilon}(0) - \theta'_0\|_2^2 \leq \varepsilon^2$, then
\begin{align} \label{error}
\sup_{t \in [0,T]} \Big( &\|\theta'_{\varepsilon}(t) - \theta'(t)\|_2^2 + \|\bar{\theta}_{\varepsilon}(t) - \bar{\theta}(t)\|_{\dot{H}^1(0,2\pi)}^2 \notag\\
&+ \|\mathbf{u}_{\varepsilon}(t) - \mathbf{u}(t)\|_{H^2_h(\Omega)}^2 + \|w_{\varepsilon}(t) - w(t)\|_{H^2_h(\Omega)}^2 \Big) \leq C \varepsilon^2,
\end{align}
where the constant $C$ depends on $T$, $\|\theta'_0\|_6$, and $\|\nabla_h \theta'_0\|_3$.
\end{theorem}

\vspace{0.1 in}

\begin{remark}
In \cite{CGT-3}, we established the global well-posedness of both weak and strong solutions to the system (\ref{fR-1})--(\ref{fR-4}) using the Galerkin method. With the well-posedness in place, we can derive an error estimate between the Galerkin approximate solution $\theta'_{\varepsilon,m}$ and the exact solution $\theta'_{\varepsilon}$ of the system (\ref{fR-1})--(\ref{fR-4}), assuming sufficiently regular initial data. Here, $m$ denotes the number of modes in the Galerkin approximation. 
By estimating the error $\|\theta'_{\varepsilon,m} - \theta'_{\varepsilon}\|_2$ and using (\ref{error}) along with the triangle inequality, one can show that there is a constant $C>0$, independent of $\varepsilon$, such that for every $\varepsilon>0$ there exists $m_0(\varepsilon) \ge 1$ for which whenever $m> m_0(\varepsilon)$ one has $\|\theta'_{\varepsilon,m} - \theta'\|_2 \leq C\varepsilon$.  This suggests that, in numerical simulations, the Galerkin solution $\theta'_{\varepsilon,m}$ of the thermally diffusive system (\ref{fR-1})--(\ref{fR-4}) can be used to approximate the exact solution $\theta'$ of the non-diffusive system (\ref{iRR-1})--(\ref{iRR-4}). For an example of such a ``double approximation'' (distinguished limit) error estimate, see the work of Cao and Titi in \cite{CT09}.
\end{remark}

\vspace{0.1 in}

\section{A Fourier Multiplier Result}

The following Fourier multiplier result will be used repeatedly throughout the paper.

\begin{proposition} \label{Fmre}
Let $a$, $b$, $c$, and $d$ be positive real numbers satisfying
$$\frac{a}{c} + \frac{b}{d} \leq 1.$$
Then the bounded sequence 
\begin{align*} 
m(\mathbf{k}) =
\begin{cases}
 \dfrac{(1+ k_3^2)^{a} (k_1^2 + k_2^2)^{b}}{(k_3^2)^c + (k_1^2 + k_2^2)^d}, & \text{if } \mathbf{k} = (k_1, k_2, k_3) \in \mathbb{Z}^3 \text{ with } k_1^2 + k_2^2 \neq 0, \\
0, & \text{if } \mathbf{k} = (k_1, k_2, k_3) \in \mathbb{Z}^3 \text{ with } k_1 = k_2 = 0,
\end{cases}
\end{align*}
is an $L^p$ Fourier multiplier on the 3D torus $\Omega = [0, 2\pi]^3$ for any $1 < p < \infty$. That is,
\begin{align} \label{app}
\left\| \sum_{\mathbf{k} \in \mathbb{Z}^3} m(\mathbf{k}) \hat{f}(\mathbf{k}) e^{i \mathbf{k} \cdot \mathbf{x}} \right\|_p \leq C_p \|f\|_p, \quad \text{for all } f \in L^p(\Omega).
\end{align}
\end{proposition}

\begin{proof}
Notice that $m(\mathbf{k})$ is defined on the integer lattice points $\mathbf{k} = (k_1, k_2, k_3) \in \mathbb{Z}^3$.
We extend $m(\mathbf{k})$ to a function $\tilde{m}(\boldsymbol{\xi})$ defined on the entire space $\mathbb{R}^3$. 
Specifically, define the bounded function $\tilde{m} : \mathbb{R}^3 \to \mathbb{R}$ by
\begin{align*}  
\tilde{m}(\boldsymbol{\xi}) =  
\begin{cases}
  \dfrac{(1 + \xi_3^2)^{a} (\xi_1^2 + \xi_2^2)^{b}}{(\xi_3^2)^c + (\xi_1^2 + \xi_2^2)^d}, & \text{if } \boldsymbol{\xi} = (\xi_1, \xi_2, \xi_3) \in \mathbb{R}^3 \text{ with } (\xi_1, \xi_2) \notin \left(-\tfrac{1}{2}, \tfrac{1}{2}\right)^2, \\
  0, & \text{if } \boldsymbol{\xi} = (\xi_1, \xi_2, \xi_3) \in \mathbb{R}^3 \text{ with } (\xi_1, \xi_2) \in \left(-\tfrac{1}{2}, \tfrac{1}{2}\right)^2.
\end{cases}
\end{align*}

It is straightforward to verify that there exists a constant $C > 0$ such that
\begin{align*}
&\left|\partial_{\xi_1} \tilde{m}(\boldsymbol{\xi})\right| \leq \frac{C}{|\xi_1|}, \quad
\left|\partial_{\xi_2} \tilde{m}(\boldsymbol{\xi})\right| \leq \frac{C}{|\xi_2|}, \quad
\left|\partial_{\xi_3} \tilde{m}(\boldsymbol{\xi})\right| \leq \frac{C}{|\xi_3|}, \\
&\left|\partial^2_{\xi_1 \xi_2} \tilde{m}(\boldsymbol{\xi})\right| \leq \frac{C}{|\xi_1 \xi_2|}, \quad
\left|\partial^2_{\xi_2 \xi_3} \tilde{m}(\boldsymbol{\xi})\right| \leq \frac{C}{|\xi_2 \xi_3|}, \quad
\left|\partial^2_{\xi_1 \xi_3} \tilde{m}(\boldsymbol{\xi})\right| \leq \frac{C}{|\xi_1 \xi_3|}, \\
&\left|\partial^3_{\xi_1 \xi_2 \xi_3} \tilde{m}(\boldsymbol{\xi})\right| \leq \frac{C}{|\xi_1 \xi_2 \xi_3|},
\end{align*}
in each dyadic rectangle $I_{j_1} \times I_{j_2} \times I_{j_3}$ for $j_1, j_2, j_3 \in \mathbb{Z}$, where the dyadic intervals are defined by
$I_j = (-2^{j+1}, -2^j) \cup (2^j, 2^{j+1}), \;j \in \mathbb{Z}$.
The constant $C$ is uniform across all dyadic rectangles. Thus, by the Marcinkiewicz Multiplier Theorem, $\tilde{m}(\boldsymbol{\xi})$ is an $L^p$ Fourier multiplier on $\mathbb{R}^3$ for any $1 < p < \infty$. For a precise statement of the Marcinkiewicz Multiplier Theorem, see, e.g., Theorem 5.2.4 on page 363 of Grafakos \cite{Grafakos}.

It is well known that there are general methods to transfer multipliers between $\mathbb{R}^n$ and the torus $\mathbb{T}^n$, known as \emph{transference of multipliers}. In particular, using Theorem 3.6.7 and Lemma 3.6.8 on page 224 of Grafakos \cite{Grafakos}, and noting that $\tilde{m}$ is an $L^p$ Fourier multiplier on $\mathbb{R}^3$ and is continuous at every integer lattice point $\mathbf{k} \in \mathbb{Z}^3$, we conclude that the sequence $\{\tilde{m}(\mathbf{k})\}_{\mathbf{k} \in \mathbb{Z}^3}$ is an $L^p$ Fourier multiplier on the 3D torus $\Omega = \mathbb T^3$. Since $\tilde{m}(\mathbf{k}) = m(\mathbf{k})$ for all $\mathbf{k} \in \mathbb{Z}^3$, it follows that $\{m(\mathbf{k})\}_{\mathbf{k} \in \mathbb{Z}^3}$ is also an $L^p$ Fourier multiplier on the 3D torus $\Omega$.
\end{proof}

\vspace{0.1 in}

\section{Global existence and uniqueness of weak solutions}

This section is devoted to proving Theorem \ref{wellp}, namely, the global existence and uniqueness of weak solutions to system (\ref{iR-1})-(\ref{iR-4}).

\subsection{Global Existence of Weak Solutions}   \label{apri}

We use a regularization method to justify the global existence of weak solutions to system~(\ref{iR-1})--(\ref{iR-4}).  
By adding a horizontal thermal diffusion term to equation~(\ref{iR-3}), we obtain system~(\ref{fRs-1})--(\ref{fRs-4}) below. 
Since a diffusion term has a smoothing effect, we regard system~(\ref{fRs-1})--(\ref{fRs-4}) as a regularization of system~(\ref{iR-1})--(\ref{iR-4}).  
In our paper~\cite{CGT-3}, we proved the global well-posedness of strong solutions to system~(\ref{fRs-1})--(\ref{fRs-4}) for initial data in \( H^1(\Omega) \).  
In this section, we aim to show that there exists a sequence \( \varepsilon \to 0 \) such that the corresponding sequence of strong solutions  
\( (\theta'_{\varepsilon}, \bar{\theta}_{\varepsilon}, \mathbf{u}_{\varepsilon}, w_{\varepsilon}) \) of system~(\ref{fRs-1})--(\ref{fRs-4})  
converges to a function \( (\theta', \bar{\theta}, \mathbf{u}, w) \), which is a weak solution of system~(\ref{iR-1})--(\ref{iR-4}),  
provided the initial data converge accordingly.

Assume \( \theta'_0 \in L^6(\Omega) \) with \( \bar{\theta'_0} = 0 \). Since \( H^1(\Omega) \) is dense in \( L^6(\Omega) \), there exists a sequence \( \theta'_{0,\varepsilon} \in H^1(\Omega) \) such that \( \bar{\theta'_{0,\varepsilon}} = 0 \) and \( \|\theta'_{0,\varepsilon} - \theta'_0\|_6 \to 0 \) as \( \varepsilon \to 0 \). According to our paper~\cite{CGT-3}, system~(\ref{fRs-1})--(\ref{fRs-4})  
admits a sequence of global strong solutions \( (\theta'_{\varepsilon}, \bar{\theta}_{\varepsilon}, \mathbf{u}_{\varepsilon}, w_{\varepsilon}) \) with initial data  
\( \theta'_{\varepsilon}(0) = \theta'_{0,\varepsilon} \in H^1(\Omega) \):
\begin{align}
& \partial_z \psi_{\varepsilon} = \theta'_{\varepsilon} + \Delta_h w_{\varepsilon}, && \text{in } L^{\infty}(0,T; H^1(\Omega)), \label{fRs-1} \\
& -\partial_z w_{\varepsilon} = \Delta_h \omega_{\varepsilon}, && \text{in } L^{\infty}(0,T; L^2(\Omega)), \label{fRs-2} \\
& \partial_t \theta'_{\varepsilon} + \mathbf{u}_{\varepsilon} \cdot \nabla_h \theta'_{\varepsilon} + w_{\varepsilon} \partial_z \bar{\theta}_{\varepsilon} = \varepsilon^2 \Delta_h \theta'_{\varepsilon}, && \text{in } L^2(\Omega \times (0,T)), \label{fRs-3} \\
& \partial_z(\overline{\theta'_{\varepsilon} w_{\varepsilon}}) = \partial_{zz} \bar{\theta}_{\varepsilon}, && \text{in } L^{\infty}(0,T; L^2(0,2\pi)), \label{fRs-4}
\end{align}
with \( \nabla_h \cdot \mathbf{u}_{\varepsilon} = 0 \). Here, $T>0$ is given arbitrarily large.

The strong solutions have regularity:  
\begin{align*}  
&\theta'_{\varepsilon} \in L^{\infty}(0,T; H^1(\Omega)) \cap C([0,T]; L^2(\Omega)) ,  \;\;  \Delta_h \theta'_{\varepsilon}, \nabla_h \partial_z \theta'_{\varepsilon} \in L^2(\Omega \times (0,T)),  \\
& \bar{\theta}_{\varepsilon} \in L^{\infty}(0,T; \dot{H}^2(0,2\pi)) , \\
&\Delta_h \mathbf{u}_{\varepsilon}, \Delta_h w_{\varepsilon} \in L^{\infty}(0,T; H^1(\Omega)) , \;\;
 \partial_z \mathbf{u}_{\varepsilon}, \partial_z w_{\varepsilon} \in L^{\infty}(0,T; L^2(\Omega)). 
\end{align*}

Our goal is to establish uniform bounds for the \( L^6 \)-norm of \( \theta'_{\varepsilon} \), ensuring that these bounds are independent of \( \varepsilon \).

\subsubsection{Solutions of the Linear Equations (\ref{fRs-1})--(\ref{fRs-2})}

Note that the linear equations (\ref{fRs-1})--(\ref{fRs-2}) connect \( \theta'_{\varepsilon} \) with \( \psi_{\varepsilon} \), \( w_{\varepsilon} \), and \( \omega_{\varepsilon} \).  
In other words, the temperature variation \( \theta'_{\varepsilon} \) is linearly related to the velocity field.  
Since equations (\ref{fRs-1})--(\ref{fRs-2}) are linear, they can be solved explicitly via the Fourier transform.  
Indeed, we have
\begin{align}  
& \widehat{u_{\varepsilon}}(\mathbf{k},t) = \left( \frac{-k_2 k_3}{k_3^2 + (k_1^2 + k_2^2)^3} \right) \widehat{\theta'_{\varepsilon}}(\mathbf{k},t), \quad
\widehat{v_{\varepsilon}}(\mathbf{k},t) = \left( \frac{k_1 k_3}{k_3^2 + (k_1^2 + k_2^2)^3} \right) \widehat{\theta'_{\varepsilon}}(\mathbf{k},t), \label{Four-uv} \\
& \widehat{w_{\varepsilon}}(\mathbf{k},t) = \left( \frac{(k_1^2 + k_2^2)^2}{k_3^2 + (k_1^2 + k_2^2)^3} \right) \widehat{\theta'_{\varepsilon}}(\mathbf{k},t), \label{Four-w}
\end{align}
for all \( \mathbf{k} = (k_1, k_2, k_3) \in \mathbb{Z}^3 \) such that \( k_1^2 + k_2^2 \neq 0 \).

\vspace{0.1 in}

\subsubsection{\( L^2 \) Estimate of \( \theta'_{\varepsilon} \)} \label{L2est}

Taking the inner product of equation~(\ref{fRs-3}) with \( \theta'_{\varepsilon} \) yields
\begin{align} \label{L2theta1}
\frac{1}{2} \frac{d}{dt} \|\theta'_{\varepsilon}\|_2^2 + \varepsilon^2 \|\nabla_h \theta'_{\varepsilon}\|_2^2 +
\int_{\Omega} w_{\varepsilon} \frac{\partial \bar{\theta}_{\varepsilon}}{\partial z} \theta'_{\varepsilon} \, dx\,dy\,dz = 0,
\quad \text{for all } t \geq 0.
\end{align}

By applying integration by parts and using equation~(\ref{fRs-4}), we rewrite the nonlinear term:
\begin{align} \label{L2theta2}
\int_{\Omega} w_{\varepsilon} \frac{\partial \bar{\theta}_{\varepsilon}}{\partial z} \theta'_{\varepsilon} \, dx\,dy\,dz 
&= 4\pi^2 \int_0^{2\pi} \overline{w_{\varepsilon} \theta'_{\varepsilon}} \frac{\partial \bar{\theta}_{\varepsilon}}{\partial z} \, dz 
= -4\pi^2 \int_0^{2\pi} \frac{\partial}{\partial z} \left( \overline{w_{\varepsilon} \theta'_{\varepsilon}} \right) \bar{\theta}_{\varepsilon} \, dz \notag \\
&= -4\pi^2 \int_0^{2\pi} \frac{\partial^2 \bar{\theta}_{\varepsilon}}{\partial z^2} \bar{\theta}_{\varepsilon} \, dz 
= 4\pi^2 \int_0^{2\pi} \left| \frac{\partial \bar{\theta}_{\varepsilon}}{\partial z} \right|^2 dz.
\end{align}

Combining equations~(\ref{L2theta1}) and~(\ref{L2theta2}), we obtain
\begin{align} \label{L2theta3}
\frac{1}{2} \frac{d}{dt} \|\theta'_{\varepsilon}\|_2^2 + \varepsilon^2 \|\nabla_h \theta'_{\varepsilon}\|_2^2 + 
4\pi^2 \int_0^{2\pi} \left| \partial_z \bar{\theta}_{\varepsilon} \right|^2 dz = 0,
\quad \text{for all } t \geq 0.
\end{align}

Integrating equation~(\ref{L2theta3}) over the interval \( [0, t] \), we obtain
\begin{align} \label{L2theta4}
\frac{1}{2} \|\theta'_{\varepsilon}(t)\|_2^2 + 
\varepsilon^2 \int_0^t \|\nabla_h \theta'_{\varepsilon}(s)\|_2^2 \, ds + 
4\pi^2 \int_0^t \int_0^{2\pi} \left| \partial_z \bar{\theta}_{\varepsilon} \right|^2 dz \, ds 
= \frac{1}{2} \|\theta'_{0,\varepsilon}\|_2^2 \leq C \|\theta'_0\|_2^2,
\end{align}
for all \( t \geq 0 \), where the initial data satisfy \( \theta'_{\varepsilon}(0) = \theta'_{0,\varepsilon} \).  
The inequality in~(\ref{L2theta4}) follows from the fact that \( \|\theta'_{0,\varepsilon} - \theta'_0\|_6 \to 0 \) as \( \varepsilon \to 0 \).

The energy estimate~(\ref{L2theta4}) shows that the \( L^2 \)-norm of \( \theta'_{\varepsilon} \) is uniformly bounded by \( C\|\theta'_0\|_2 \) for all \( t \geq 0 \).  
Moreover, the norm \( \|\partial_z \bar{\theta}_{\varepsilon}\|_{L^2((0,2\pi) \times (0,t))} \) is also uniformly bounded by \( C\|\theta'_0\|_2 \) for all \( t \geq 0 \).  
Note that these bounds are independent of \( \varepsilon \).

\vspace{0.1 in}

\subsubsection{\( L^3 \) Estimate of \( \theta'_{\varepsilon} \)}

This subsection is devoted to estimating the \( L^3 \)-norm of \( \theta'_{\varepsilon} \).

Multiplying equation~(\ref{fRs-3}) by \( |\theta'_{\varepsilon}|\theta'_{\varepsilon} \) and integrating over \( \Omega \), we obtain
\begin{align} \label{L3theta1}
\frac{1}{3} \frac{d}{dt} \|\theta'_{\varepsilon}\|_3^3 
+ 2\varepsilon^2 \int_{\Omega} |\nabla_h \theta'_{\varepsilon}|^2 |\theta'_{\varepsilon}| \, dx\,dy\,dz 
+ \int_{\Omega} w_{\varepsilon} \frac{\partial \bar{\theta}_{\varepsilon}}{\partial z} |\theta'_{\varepsilon}|\theta'_{\varepsilon} \, dx\,dy\,dz = 0,
\end{align}
for all $t \geq 0$.

From equation~(\ref{fRs-4}), we have \( \frac{\partial \bar{\theta}_{\varepsilon}}{\partial z} = \overline{\theta'_{\varepsilon} w_{\varepsilon}} + c(t) \),  
where \( c(t) \) depends only on time. Since \( \bar{\theta}_{\varepsilon} \) is periodic on \( [0,2\pi] \), it follows that
\[
0 = \bar{\theta}_{\varepsilon}(2\pi) - \bar{\theta}_{\varepsilon}(0) = \int_0^{2\pi} \overline{\theta'_{\varepsilon} w_{\varepsilon}} \, dz + 2\pi c(t),
\]
which implies
\[
c(t) = -\frac{1}{2\pi} \int_0^{2\pi} \overline{\theta'_{\varepsilon} w_{\varepsilon}} \, dz.
\]
Therefore,
\begin{align} \label{L3theta1'}
\frac{\partial \bar{\theta}_{\varepsilon}}{\partial z} 
= \overline{\theta'_{\varepsilon} w_{\varepsilon}} - \frac{1}{2\pi} \int_0^{2\pi} \overline{\theta'_{\varepsilon} w_{\varepsilon}} \, dz 
= \frac{1}{4\pi^2} \int_{[0,2\pi]^2} \theta'_{\varepsilon} w_{\varepsilon} \, dx\,dy 
- \frac{1}{8\pi^3} \int_{\Omega} \theta'_{\varepsilon} w_{\varepsilon} \, dx\,dy\,dz.
\end{align}

Using~(\ref{L3theta1'}) and Hölder's inequality, we deduce
\begin{align} \label{L3theta2}
\left| \int_{\Omega} w_{\varepsilon} \frac{\partial \bar{\theta}_{\varepsilon}}{\partial z} |\theta'_{\varepsilon}|\theta'_{\varepsilon} \, dx\,dy\,dz \right| 
\leq C \left( \sup_{z \in [0,2\pi]} \int_{[0,2\pi]^2} |w_{\varepsilon}|^3 \, dx\,dy \right)^{2/3} \|\theta'_{\varepsilon}\|_3^3.
\end{align}

We now estimate the term \( \sup_{z \in [0,2\pi]} \int_{[0,2\pi]^2} |w_{\varepsilon}|^3 \, dx\,dy \) using Sobolev embeddings in the horizontal and vertical directions,  
along with the explicit relation between the velocity field and \( \theta'_{\varepsilon} \) given in~(\ref{Four-uv})--(\ref{Four-w}).

Let \( A = -\Delta_h \). Recall the two-dimensional embedding \( H^{1/3}([0,2\pi]^2) \hookrightarrow L^3([0,2\pi]^2) \) and the one-dimensional embedding
\( H^s(0,2\pi) \hookrightarrow L^{\infty}(0,2\pi) \) for \( s > 1/2 \), in particular \( H^{5/9}(0,2\pi) \hookrightarrow L^{\infty}(0,2\pi) \).  
Applying these successively, we obtain
\begin{align} \label{L3theta3'}
&\sup_{z \in [0,2\pi]} \left( \int_{[0,2\pi]^2} |w_{\varepsilon}(x,y,z)|^3 \, dx\,dy \right)^{1/3} 
= \sup_{z \in [0,2\pi]} \|w_{\varepsilon}(z)\|_{L^3([0,2\pi]^2)}           \notag \\
&\leq C \sup_{z \in [0,2\pi]} \|w_{\varepsilon}(z)\|_{H^{1/3}([0,2\pi]^2)}   \leq C \left(\int_{[0,2\pi]^2}\sup_{z\in [0,2\pi]}   |  A^{1/6} w_{\varepsilon}|^2 \, dx\,dy \right)^{1/2}         \notag\\
&\leq C \left( \int_{[0,2\pi]^2} \|A^{1/6} w_{\varepsilon}\|_{H^{5/9}(0,2\pi)}^2 \, dx\,dy \right)^{1/2} = C \|(I - \partial_{zz})^{5/18} A^{1/6} w_{\varepsilon}\|_2.
\end{align}

From~(\ref{Four-w}), we have
\begin{align} \label{L3theta3''}
\|(I - \partial_{zz})^{5/18} A^{1/6} w_{\varepsilon}\|_2^2 
\leq \sum_{\substack{\mathbf{k} = (k_1,k_2,k_3) \in \mathbb{Z}^3 \\ k_1^2 + k_2^2 \neq 0}} 
\frac{(1 + k_3^2)^{5/9} (k_1^2 + k_2^2)^{13/3}}{k_3^4 + (k_1^2 + k_2^2)^6} |\widehat{\theta'_{\varepsilon}}(\mathbf{k})|^2 
\leq C \|\theta'_{\varepsilon}\|_2^2.
\end{align}

Combining~(\ref{L3theta3'}) and~(\ref{L3theta3''}), we obtain
\begin{align} \label{L3theta3b}
\sup_{z \in [0,2\pi]} \left( \int_{[0,2\pi]^2} |w_{\varepsilon}|^3 \, dx\,dy \right)^{1/3} \leq C \|\theta'_{\varepsilon}\|_2.
\end{align}

Substituting~(\ref{L3theta3b}) into~(\ref{L3theta2}), we find
\begin{align} \label{L3theta3}
\left| \int_{\Omega} w_{\varepsilon} \frac{\partial \bar{\theta}_{\varepsilon}}{\partial z} |\theta'_{\varepsilon}|\theta'_{\varepsilon} \, dx\,dy\,dz \right| 
\leq C \|\theta'_{\varepsilon}\|_2^2 \|\theta'_{\varepsilon}\|_3^3 
\leq C \|\theta'_0\|_2^2 \|\theta'_{\varepsilon}\|_3^3,
\end{align}
where we used the uniform bound \( \|\theta'_{\varepsilon}(t)\|_2 \leq C\|\theta'_0\|_2 \) for all \( t \geq 0 \), from~(\ref{L2theta4}).

Therefore, from~(\ref{L3theta1}) and~(\ref{L3theta3}), we conclude
\begin{align} \label{L3theta4}
\frac{d}{dt} \|\theta'_{\varepsilon}\|_3^3 \leq C \|\theta'_0\|_2^2 \|\theta'_{\varepsilon}\|_3^3,
\quad \text{for all } t \geq 0.
\end{align}

Applying Grönwall's inequality yields
\begin{align} \label{L3theta5}
\|\theta'_{\varepsilon}(t)\|_3^3 \leq \|\theta'_{0,\epsilon}\|_3^3 \exp\left( C \|\theta'_0\|_2^2 t \right) 
\leq C\|\theta'_0 \|_3^3 \exp\left( C \|\theta'_0\|_2^2 t \right),
\quad \text{for all } t
\geq 0.
\end{align}

This shows that the \( L^3 \)-norm of \( \theta'_{\varepsilon} \) is bounded above by a quantity that grows at most exponentially in time,  
with the bound depending only on the initial data \( \theta'_0 \), and crucially, it is \emph{independent} of \( \varepsilon \).

\vspace{0.1 in}

\subsubsection{\( L^6 \) Estimate of \( \theta'_{\varepsilon} \)} \label{sec-lp}

To find a uniform bound for the \( L^6 \)-norm of \( \theta'_{\varepsilon} \), we use the bound for the \( L^3 \)-norm established in~(\ref{L3theta5}).

Multiplying equation~(\ref{fRs-3}) by \( (\theta_{\varepsilon}')^5 \) and integrating over \( \Omega \), we obtain
\begin{align} \label{Lptheta1}   
\frac{1}{6} \frac{d}{dt} \|\theta'_{\varepsilon}\|_6^6
+ 5\varepsilon^2 \int_{\Omega} |\nabla_h \theta'_{\varepsilon}|^2 |\theta'_{\varepsilon}|^4 \, dx\,dy\,dz 
+ \int_{\Omega} w_{\varepsilon} \frac{\partial \bar{\theta}_{\varepsilon}}{\partial z} (\theta_{\varepsilon}')^5 \, dx\,dy\,dz = 0,
\end{align}
for all \( t \geq 0 \).

To estimate the last term in~(\ref{Lptheta1}), we proceed analogously to the previous subsection.  
Using~(\ref{L3theta1'}) and Hölder's inequality, we obtain
\begin{align} \label{Lptheta2}
\left| \int_{\Omega} w_{\varepsilon} \frac{\partial \bar{\theta}_{\varepsilon}}{\partial z} (\theta'_{\varepsilon})^5 \, dx\,dy\,dz \right| 
\leq C \sup_{z \in [0,2\pi]} \left( \int_{[0,2\pi]^2} |w_{\varepsilon}|^6 \, dx\,dy \right)^{1/3} \|\theta'_{\varepsilon}\|_6^6.
\end{align}

Our goal is to show that the quantity \( \sup_{z \in [0,2\pi]} \int_{[0,2\pi]^2} |w_{\varepsilon}|^6 \, dx\,dy \)  
is bounded in terms of the \( L^3 \)-norm of \( \theta'_{\varepsilon} \), and then apply~(\ref{L3theta5}) to control it.  
This involves using anisotropic Sobolev embeddings and Fourier multiplier theory.

Let \( A = -\Delta_h \). Recall the two-dimensional Sobolev embedding \( W^{\frac{1}{3},3}([0,2\pi]^2) \hookrightarrow L^6([0,2\pi]^2) \). Then,
\begin{align} \label{Lptheta2'}
&\sup_{z \in [0,2\pi]} \left( \int_{[0,2\pi]^2} |w_{\varepsilon}|^6 \, dx\,dy \right)^{1/6} 
= \sup_{z \in [0,2\pi]} \|w_{\varepsilon}(z)\|_{L^6([0,2\pi]^2)} \leq C \sup_{z \in [0,2\pi]} \|w_{\varepsilon}(z)\|_{W^{\frac{1}{3},3}([0,2\pi]^2)} \notag \\
&= C  \sup_{z \in [0,2\pi]}    \left( \int_{[0,2\pi]^2}  |A^{1/6} w_{\varepsilon}|^3 \, dx\,dy \right)^{1/3} \leq  C \left( \int_{[0,2\pi]^2} \sup_{z \in [0,2\pi]} |A^{1/6} w_{\varepsilon}|^3 \, dx\,dy \right)^{1/3}.
\end{align}

Also, from the one-dimensional embedding \( W^{s,3}(0,2\pi) \hookrightarrow L^{\infty}(0,2\pi) \) for \( s > 1/3 \),  
we use \( W^{\frac{1}{2},3}(0,2\pi) \hookrightarrow L^{\infty}(0,2\pi) \) to get
\begin{align} \label{Lptheta2''}
\sup_{z \in [0,2\pi]} |A^{1/6} w_{\varepsilon}| 
\leq C \|A^{1/6} w_{\varepsilon}\|_{W^{\frac{1}{2},3}(0,2\pi)} 
= C \left( \int_0^{2\pi} |(I - \partial_{zz})^{1/4} A^{1/6} w_{\varepsilon}|^3 \, dz \right)^{1/3}.
\end{align}

Combining~(\ref{Lptheta2'}) and~(\ref{Lptheta2''}), we obtain
\begin{align} \label{Lptheta3}
\sup_{z \in [0,2\pi]} \left( \int_{[0,2\pi]^2} |w_{\varepsilon}|^6 \, dx\,dy \right)^{1/6} 
\leq C \| (I - \partial_{zz})^{1/4} A^{1/6} w_{\varepsilon} \|_3.
\end{align}

We claim that \( \| (I - \partial_{zz})^{1/4} A^{1/6} w_{\varepsilon} \|_3 \leq C \|\theta'_{\varepsilon}\|_3 \).  
Indeed, from~(\ref{Four-w}), we have
\[
(I - \partial_{zz})^{1/4} A^{1/6} w_{\varepsilon} = \sum_{\mathbf{k} \in \mathbb{Z}^3} m(\mathbf{k}) \widehat{\theta'_{\varepsilon}}(\mathbf{k}) e^{i \mathbf{k} \cdot \mathbf{x}},
\]
where
\[
m(\mathbf{k}) = 
\begin{cases}
\frac{(1 + k_3^2)^{1/4} (k_1^2 + k_2^2)^{13/6}}{k_3^2 + (k_1^2 + k_2^2)^3}, & \text{if } k_1^2 + k_2^2 \neq 0, \\
0, & \text{if } k_1 = k_2 = 0,
\end{cases}
\]
for $\mathbf{k}=(k_1, k_2, k_3) \in \mathbb Z^3$.

By Proposition~\ref{Fmre}, \( m(\mathbf{k}) \) is an \( L^p \) Fourier multiplier for all \( 1 < p < \infty \). Hence,
\begin{align} \label{Lptheta3'}
\|(I - \partial_{zz})^{1/4} A^{1/6} w_{\varepsilon}\|_3 
\leq C \|\theta'_{\varepsilon}\|_3.
\end{align}

Combining~(\ref{Lptheta3}) and~(\ref{Lptheta3'}), we conclude
\begin{align} \label{Lptheta3''}
\sup_{z \in [0,2\pi]} \left( \int_{[0,2\pi]^2} |w_{\varepsilon}|^6 \, dx\,dy \right)^{1/6} 
\leq C \|\theta'_{\varepsilon}\|_3.
\end{align}

Substituting into~(\ref{Lptheta2}), we obtain the estimate for the nonlinear term:
\begin{align} \label{Lptheta4b}
\left| \int_{\Omega} w_{\varepsilon} \frac{\partial \bar{\theta}_{\varepsilon}}{\partial z} (\theta'_{\varepsilon})^5 \, dx\,dy\,dz \right| 
\leq C \|\theta'_{\varepsilon}\|_3^2 \|\theta'_{\varepsilon}\|_6^6.
\end{align}

Therefore, from~(\ref{Lptheta1}) and~(\ref{Lptheta4b}), we get
\begin{align} \label{Lptheta4}
\frac{d}{dt} \|\theta'_{\varepsilon}\|_6^6 \leq C \|\theta'_{\varepsilon}\|_3^2 \|\theta'_{\varepsilon}\|_6^6, \quad \text{for all } t \geq 0.
\end{align}

Applying Grönwall's inequality and using~(\ref{L3theta5}), we derive
\begin{align} \label{Lptheta5}
\|\theta'_{\varepsilon}(t)\|_6^6 
\leq \|\theta'_{0,\varepsilon}\|_6^6 \exp\left( C \int_0^t \|\theta'_{\varepsilon}(s)\|_3^2 \, ds \right) 
\leq C \|\theta'_0\|_6^6 \exp\left( C \|\theta'_0\|_3^2 \int_0^t e^{C \|\theta'_0\|_2^2 s} \, ds \right),
\end{align}
for all $t\geq 0$. This shows that the \( L^6 \)-norm of \( \theta'_{\varepsilon} \) is bounded above by a function that grows double-exponentially in time.  
Importantly, this bound is \emph{independent} of \( \varepsilon \).

Additionally, by using (\ref{Four-uv})--(\ref{Four-w}) and (\ref{Lptheta5}), and applying Proposition~\ref{Fmre}, we obtain the uniform bound:
\begin{align} \label{Lpuvw}
\|\mathbf{u}_{\varepsilon}(t)\|_6 + \|w_{\varepsilon}(t)\|_6 \leq C \|\theta'_{\varepsilon}(t)\|_6 \leq C(\|\theta'_0\|_6, t).
\end{align}

\vspace{0.1 in}

\subsubsection{Passing to the Limit} \label{passage}

In this subsection, we show that one can extract a subsequence of \( (\theta'_{\varepsilon}, \bar{\theta}_{\varepsilon}, \mathbf{u}_{\varepsilon}, w_{\varepsilon}) \) that converges to a weak solution of system~(\ref{iR-1})--(\ref{iR-4}).

Let \( T > 0 \) be arbitrary. Due to the uniform bound~(\ref{Lptheta5}), there exists a function \( \theta' \) and a subsequence of \( \varepsilon \) such that
\begin{align}
\theta'_{\varepsilon} \rightharpoonup \theta' \quad \text{weakly$^*$ in } L^{\infty}(0,T; L^6(\Omega)), \quad \text{as } \varepsilon \to 0. \label{ex2}
\end{align}

To pass to the limit in the nonlinear terms, we need strong convergence of \( \mathbf{u}_{\varepsilon} \) and \( w_{\varepsilon} \).  
To this end, we estimate the time derivative \( \partial_t \theta'_{\varepsilon} \).  
Since \( \nabla_h \cdot \mathbf{u}_{\varepsilon} = 0 \), we have, for any \( \varphi \in H^1_h(\Omega) \),
\begin{align*}
\int_{\Omega} (\mathbf{u}_{\varepsilon} \cdot \nabla_h \theta'_{\varepsilon}) \varphi \, dx\,dy\,dz 
&= -\int_{\Omega} (\mathbf{u}_{\varepsilon} \cdot \nabla_h \varphi) \theta'_{\varepsilon} \, dx\,dy\,dz \\
&\leq \|\mathbf{u}_{\varepsilon}\|_3 \|\theta'_{\varepsilon}\|_6 \|\nabla_h \varphi\|_2 
\leq C(T, \|\theta'_0\|_6) \|\nabla_h \varphi\|_2,
\end{align*}
where we used~(\ref{Lptheta5}) and~(\ref{Lpuvw}). Therefore,
\begin{align} \label{ex5}
\sup_{t \in [0,T]} \|\mathbf{u}_{\varepsilon} \cdot \nabla_h \theta'_{\varepsilon}\|_{(H^1_h(\Omega))'} \leq C(T, \|\theta'_0\|_6).
\end{align}

Next, we estimate \( w_{\varepsilon} \, \partial_z \bar{\theta}_{\varepsilon} \).  
Using~(\ref{Four-w}), we obtain \( \|(I - \partial_{zz})^{1/3} w_{\varepsilon}\|_2 \leq C \|\theta'_{\varepsilon}\|_2 \).  
Then, for any \( \phi \in L^2(\Omega \times (0,T)) \), we have
\begin{align*}
\int_0^T \int_{\Omega} (w_{\varepsilon} \, \partial_z \bar{\theta}_{\varepsilon}) \phi \, dx\,dy\,dz\,dt 
&\leq \int_0^T \left( \int_{[0,2\pi]^2} \sup_{z \in [0,2\pi]} |w_{\varepsilon}|^2 \, dx\,dy \right)^{1/2} \|\phi\|_2 \left( \int_0^{2\pi} |\partial_z \bar{\theta}_{\varepsilon}|^2 \, dz \right)^{1/2} dt \\
&\leq  C \int_0^T \|  (I - \partial_{zz})^{1/3} w_{\varepsilon}\|_2 \|\phi\|_2 \left( \int_0^{2\pi} |\partial_z \bar{\theta}_{\varepsilon}|^2 \, dz \right)^{1/2} dt \\
&\leq C \int_0^T \|\theta'_{\varepsilon}\|_2 \|\phi\|_2 \left( \int_0^{2\pi} |\partial_z \bar{\theta}_{\varepsilon}|^2 \, dz \right)^{1/2} dt \\
&\leq C \|\theta'_0\|_2 \|\phi\|_{L^2(\Omega \times (0,T))} \|\partial_z \bar{\theta}_{\varepsilon}\|_{L^2((0,2\pi) \times (0,T))} \\
&\leq C \|\theta'_0\|_2^2 \|\phi\|_{L^2(\Omega \times (0,T))},
\end{align*}
where we used~(\ref{L2theta4}). Hence,
\begin{align} \label{ex6}
\|w_{\varepsilon} \, \partial_z \bar{\theta}_{\varepsilon}\|_{L^2(\Omega \times (0,T))} \leq C \|\theta'_0\|_2^2.
\end{align}

Since \( \varepsilon^2 \int_0^T \|\nabla_h \theta'_{\varepsilon}(s)\|_2^2 \, ds \leq C \|\theta'_0\|_2^2 \) by~(\ref{L2theta4}),  
we obtain that \( \varepsilon^2 \Delta_h \theta'_{\varepsilon} \) is uniformly bounded in \( L^2(0,T; (H^1_h(\Omega))') \).  
Therefore, from~(\ref{fRs-3}),~(\ref{ex5}), and~(\ref{ex6}), we conclude that \( \partial_t \theta'_{\varepsilon} \) is uniformly bounded in \( L^2(0,T; (H^1_h(\Omega))') \).

Moreover, using the explicit relations~(\ref{Four-uv}) and~(\ref{Four-w}), we obtain
\[
\|\partial_t \mathbf{u}_{\varepsilon}\|_2^2 + \|\partial_t w_{\varepsilon}\|_2^2 \leq C \|\partial_t \theta'_{\varepsilon}\|_{(H^1_h(\Omega))'}^2,
\]
so \( \partial_t \mathbf{u}_{\varepsilon} \) and \( \partial_t w_{\varepsilon} \) are uniformly bounded in \( L^2(\Omega \times (0,T)) \).

Additionally, from~(\ref{Four-uv}) and~(\ref{Four-w}), we have
\begin{align*}
\|\Delta_h \mathbf{u}_{\varepsilon}\|_2^2 + \|\Delta_h w_{\varepsilon}\|_2^2 &\leq C \|\theta'_{\varepsilon}\|_2^2, \\
\|(I - \partial_{zz})^{1/3} \mathbf{u}_{\varepsilon}\|_2^2 + \|(I - \partial_{zz})^{1/3} w_{\varepsilon}\|_2^2 &\leq C \|\theta'_{\varepsilon}\|_2^2.
\end{align*}

Thus, \( \mathbf{u}_{\varepsilon} \) and \( w_{\varepsilon} \) are uniformly bounded in \( L^{\infty}(0,T; H^{2/3}(\Omega)) \).  
By the Aubin–Lions compactness theorem, and extracting a subsequence if necessary, we obtain the strong convergence
\begin{align} \label{ex7}
(\mathbf{u}_{\varepsilon}, w_{\varepsilon}) \to (\mathbf{u}, w) \quad \text{strongly in } L^2(\Omega \times (0,T)), \quad \text{as } \varepsilon \to 0.
\end{align}

Using the weak convergence~(\ref{ex2}) and the strong convergence~(\ref{ex7}), we can pass to the limit as \( \varepsilon \to 0 \) in the weak formulation of system~(\ref{fRs-1})--(\ref{fRs-4}).  
This yields the existence of a global weak solution to system~(\ref{iR-1})--(\ref{iR-4}) via a standard argument.

\vspace{0.1 in}

\subsection{Uniqueness of Weak Solutions} \label{sec-uni}

Previously, we established the global existence of weak solutions to system~(\ref{iR-1})--(\ref{iR-4}).  
In this subsection, we justify the uniqueness of such weak solutions.  
The key idea is to perform an energy estimate using a norm weaker than the \( L^2(\Omega) \)-norm; specifically, we use the norm of the dual space \( (H^1_h(\Omega))' \).  
This approach allows us to accommodate weak solutions with initial data in \( L^6(\Omega) \). This idea is adapted from Yudovich \cite{Yu} and Larios et al. \cite{LLT}. 

Let \( (\theta'_1, \bar{\theta}_1, \mathbf{u}_1, w_1) \) and \( (\theta'_2, \bar{\theta}_2, \mathbf{u}_2, w_2) \) be two weak solutions on \( [0,T] \)  
satisfying equations~(\ref{iR-1})--(\ref{iR-4}) and the regularity conditions stated in Theorem~\ref{wellp}.  
Assume that \( \theta'_1(0) = \theta'_2(0) = \theta'_0 \in L^6(\Omega) \).

Let \( \theta' = \theta'_1 - \theta'_2 \), \( \bar{\theta} = \bar{\theta}_1 - \bar{\theta}_2 \), \( \mathbf{u} = \mathbf{u}_1 - \mathbf{u}_2 \),  
\( w = w_1 - w_2 \), \( \omega = \omega_1 - \omega_2 \), and \( \psi = \psi_1 - \psi_2 \). Then, we have
\begin{align}
& \partial_z \psi = \theta' + \Delta_h w, && \text{in } L^{\infty}(0,T; L^6(\Omega)); \label{inun-1} \\
& -\partial_z w = \Delta_h \omega, && \text{in } L^{\infty}(0,T; (H^1_h(\Omega))'); \label{inun-2} \\
& \partial_t \theta' + \mathbf{u} \cdot \nabla_h \theta'_1 + \mathbf{u}_2 \cdot \nabla_h \theta' + w \, \partial_z \bar{\theta}_1 + w_2 \, \partial_z \bar{\theta} = 0, && \text{in } L^2(0,T; (H^1_h(\Omega))'); \label{inun-3} \\
& \partial_z(\overline{\theta' w_2}) + \partial_z(\overline{\theta'_1 w}) = \partial_{zz} \bar{\theta}, && \text{in } L^2(0,T; H^{-1}(0,2\pi)), \label{inun-4}
\end{align}
with \( \nabla_h \cdot \mathbf{u} = 0 \). Note that \( \theta'(0) = 0 \), since \( \theta'_1 \) and \( \theta'_2 \) share the same initial data \( \theta'_0 \).

Using the linear equations~(\ref{inun-1})--(\ref{inun-2}), we obtain that the Fourier coefficients of \( (\mathbf{u}, w) = (u, v, w) \) are related to those of \( \theta' \) by:
\begin{align}
& \widehat{u}(\mathbf{k}, t) = \left( \frac{-k_2 k_3}{k_3^2 + (k_1^2 + k_2^2)^3} \right) \widehat{\theta'}(\mathbf{k}, t), \quad
\widehat{v}(\mathbf{k}, t) = \left( \frac{k_1 k_3}{k_3^2 + (k_1^2 + k_2^2)^3} \right) \widehat{\theta'}(\mathbf{k}, t), \label{F-uv} \\
& \widehat{w}(\mathbf{k}, t) = \left( \frac{(k_1^2 + k_2^2)^2}{k_3^2 + (k_1^2 + k_2^2)^3} \right) \widehat{\theta'}(\mathbf{k}, t), \label{F-w}
\end{align}
for all \( \mathbf{k} = (k_1, k_2, k_3) \in \mathbb{Z}^3 \) such that \( k_1^2 + k_2^2 \neq 0 \).

Recall the notation \( A = -\Delta_h \). Take the duality pairing of (\ref{inun-3}) with \( A^{-1} \theta' \) in \( (H^1_h(\Omega))' \times H^1_h(\Omega) \).  
Also, take the duality pairing of (\ref{inun-4}) with \( \bar{\theta} \) in \( H^{-1}(0,2\pi) \times \dot{H}^1(0,2\pi) \).  
Adding the results yields
\begin{align}  \label{ntd1}
&\frac{1}{2} \frac{d}{dt} \| A^{-1/2} \theta' \|_2^2 + \| \bar{\theta}_z \|_{L^2(0,2\pi)}^2 \notag\\
&= - \left\langle \mathbf{u} \cdot \nabla_h \theta'_1, A^{-1} \theta' \right\rangle
- \left\langle \mathbf{u}_2 \cdot \nabla_h \theta', A^{-1} \theta' \right\rangle  - \left\langle w \, \partial_z \bar{\theta}_1, A^{-1} \theta' \right\rangle
- \left\langle w_2 \bar{\theta}_z, A^{-1} \theta' \right\rangle \notag \\
&\quad + \int_0^{2\pi} (\overline{\theta' w_2}) \, \bar{\theta}_z \, dz
+ \int_0^{2\pi} (\overline{\theta'_1 w}) \, \bar{\theta}_z \, dz. 
\end{align}

In the following, we estimate each term on the right-hand side of (\ref{ntd1}).

Since \( \nabla_h \cdot \mathbf{u} = 0 \), we have
\begin{align} \label{first1}
- \left\langle \mathbf{u} \cdot \nabla_h \theta'_1, A^{-1} \theta' \right\rangle 
= \int_{\Omega} (\mathbf{u} \cdot \nabla_h A^{-1} \theta') \, \theta'_1 \, dx \, dy \, dz 
\leq C \| \mathbf{u} \|_{10/3} \| A^{-1/2} \theta' \|_2 \| \theta'_1 \|_5.
\end{align}

We claim that \( \| \mathbf{u} \|_{10/3} \) is bounded by \( \| A^{-1/2} \theta' \|_2 \). Indeed,  
by using the 2D embedding \( H^{2/5}([0,2\pi]^2) \hookrightarrow L^{10/3}([0,2\pi]^2) \),  
the 1D embedding \( H^{1/5}(0,2\pi) \hookrightarrow L^{10/3}(0,2\pi) \),  
and Minkowski's integral inequality, we obtain
\begin{align} \label{first2}
\| \mathbf{u} \|_{10/3} 
&= \left( \int_0^{2\pi} \| \mathbf{u} \|_{L^{10/3}([0,2\pi]^2)}^{10/3} \, dz \right)^{3/10} \notag \\
&\leq C \left( \int_0^{2\pi} \| A^{1/5} \mathbf{u} \|_{L^2([0,2\pi]^2)}^{10/3} \, dz \right)^{3/10} \notag \\
&\leq C \left( \int_{[0,2\pi]^2} \| A^{1/5} \mathbf{u} \|_{L^{10/3}(0,2\pi)}^2 \, dx \, dy \right)^{1/2} \notag \\
&\leq C \left( \int_{[0,2\pi]^2} \| (I - \partial_{zz})^{1/10} A^{1/5} \mathbf{u} \|_{L^2(0,2\pi)}^2 \, dx \, dy \right)^{1/2} \notag \\
&= C \| (I - \partial_{zz})^{1/10} A^{1/5} \mathbf{u} \|_2 
\leq C \| A^{-1/2} \theta' \|_2,
\end{align}
where the last inequality follows from (\ref{F-uv}), which relates the Fourier coefficients of \( \mathbf{u} \) and \( \theta' \).

Therefore, combining (\ref{first1}) and (\ref{first2}), we obtain
\begin{align} \label{first3}
- \left\langle \mathbf{u} \cdot \nabla_h \theta'_1, \Delta_h^{-1} \theta' \right\rangle
\leq C \| \theta'_1 \|_5 \| A^{-1/2} \theta' \|_2^2.
\end{align}

Also, since \( \nabla_h \cdot \mathbf{u}_2 = 0 \), we have
\begin{align} \label{Second1}
- \left\langle \mathbf{u}_2 \cdot \nabla_h \theta', A^{-1} \theta' \right\rangle
&= \int_{\Omega} (\mathbf{u}_2 \cdot \nabla_h A^{-1} \theta') \, \theta' \, dx \, dy \, dz \notag \\
&= \int_{\Omega} (\mathbf{u}_2 \cdot \nabla_h A^{-1} \theta') \, (-\Delta_h A^{-1} \theta') \, dx \, dy \, dz \notag \\
&\leq C \| \nabla_h \mathbf{u}_2 \|_{\infty} \| \nabla_h A^{-1} \theta' \|_2^2 
\leq C \| \nabla_h \mathbf{u}_2 \|_{\infty} \| A^{-1/2} \theta' \|_2^2.
\end{align}

Since \( W^{\frac{1}{5},6}(0,2\pi) \hookrightarrow L^{\infty}(0,2\pi) \) and \( W^{\frac{2}{5},6}([0,2\pi]^2) \hookrightarrow L^{\infty}([0,2\pi]^2) \), we deduce
\begin{align} \label{Second2}
\| \nabla_h \mathbf{u}_2 \|_{\infty} 
&\leq C \| A^{1/5} (I - \partial_{zz})^{1/10} \nabla_h \mathbf{u}_2 \|_6 
\leq C \| \theta'_2 \|_6,
\end{align}
due to the relation between the Fourier coefficients of \( \mathbf{u}_2 \) and \( \theta'_2 \) as in (\ref{F-uv}),  
and the $L^p$ Fourier multiplier result: Proposition \ref{Fmre}.

By (\ref{Second1}) and (\ref{Second2}), we obtain
\begin{align} \label{Second3}
- \left\langle \mathbf{u}_2 \cdot \nabla_h \theta', A^{-1} \theta' \right\rangle 
\leq C \| \theta'_2 \|_6 \| A^{-1/2} \theta' \|_2^2.
\end{align}

Next, we estimate the third term on the right-hand side of (\ref{ntd1}) as follows:
\begin{align} \label{ntd3'}
- \left\langle w \, \partial_z \bar{\theta}_1, A^{-1} \theta' \right\rangle
&= -\int_0^{2\pi} \partial_z \bar{\theta}_1 \left( \int_{[0,2\pi]^2} w \, A^{-1} \theta' \, dx \, dy \right) dz \notag \\
&\leq C \int_0^{2\pi} \left| \partial_z \bar{\theta}_1 \right| \, \| A^{-1/2} w \|_{L^2([0,2\pi]^2)} \, \| A^{-1/2} \theta' \|_{L^2([0,2\pi]^2)} \, dz \notag \\
&\leq C \| A^{-1/2} \theta' \|_2 \left( \int_0^{2\pi} (\partial_z \bar{\theta}_1)^2 dz \right)^{1/2}
\left( \int_{[0,2\pi]^2} \sup_{z \in [0,2\pi]} |A^{-1/2} w|^2 dx \, dy \right)^{1/2} \notag \\
&\leq C \| A^{-1/2} \theta' \|_2 \, \| \partial_z \bar{\theta}_1 \|_{L^2(0,2\pi)} \, \| (I - \partial_{zz})^{1/3} (A^{-1/2} w) \|_2,
\end{align}
where we have used the one-dimensional embedding \( H^s(0,2\pi) \hookrightarrow L^{\infty}(0,2\pi) \) for \( s > 1/2 \).

From (\ref{F-w}), we obtain
\begin{align} \label{z82}
\| (I - \partial_{zz})^{1/3} (A^{-1/2} w) \|_2^2 \leq C \| A^{-1/2} \theta' \|_2^2.
\end{align}

Using (\ref{ntd3'}) and (\ref{z82}), it follows that
\begin{align} \label{ntd3}
- \left\langle w \, \partial_z \bar{\theta}_1, A^{-1} \theta' \right\rangle 
\leq C \| A^{-1/2} \theta' \|_2^2 \, \| \partial_z \bar{\theta}_1 \|_{L^2(0,2\pi)}.
\end{align}

Now we consider the fourth term on the right-hand side of (\ref{ntd1}). Indeed,
\begin{align} \label{third1}
- \left\langle w_2 \, \bar{\theta}_z, A^{-1} \theta' \right\rangle 
&= - \int_0^{2\pi} \bar{\theta}_z \left( \int_{[0, 2\pi]^2} w_2 \, A^{-1} \theta' \, dx \, dy \right) dz \notag \\
&\leq C \int_0^{2\pi} |\bar{\theta}_z| \, \| w_2 \|_{L^2([0,2\pi]^2)} \, \| A^{-1} \theta' \|_{L^2([0,2\pi]^2)} \, dz \notag \\
&\leq C \| \bar{\theta}_z \|_{L^2(0,2\pi)} \, \| A^{-1} \theta' \|_2 \left( \sup_{z \in [0,2\pi]} \| w_2 \|_{L^2([0,2\pi]^2)} \right).
\end{align}

Similar to (\ref{ntd3'})--(\ref{z82}), we have
\begin{align} \label{third3}
\sup_{z \in [0,2\pi]} \| w_2 \|_{L^2([0,2\pi]^2)}^2 
&\leq \int_{[0,2\pi]^2} \sup_{z \in [0,2\pi]} |w_2|^2 \, dx \, dy 
\leq \| (I - \partial_{zz})^{1/3} w_2 \|_2^2 
\leq C \| \theta'_2 \|_2^2.
\end{align}

Combining (\ref{third1}) and (\ref{third3}) yields
\begin{align} \label{third5}
- \left\langle w_2 \, \bar{\theta}_z, A^{-1} \theta' \right\rangle 
&\leq C \| \bar{\theta}_z \|_{L^2(0,2\pi)} \, \| A^{-1} \theta' \|_2 \, \| \theta'_2 \|_2 \notag \\
&\leq \frac{1}{6} \| \bar{\theta}_z \|_{L^2(0,2\pi)}^2 + C \| A^{-1} \theta' \|_2^2 \, \| \theta'_2 \|_2^2.
\end{align}

Next, we estimate the fifth term on the right-hand side of (\ref{ntd1}) as follows:
\begin{align} \label{fifth1}
\int_0^{2\pi} (\overline{\theta' w_2}) \, \bar{\theta}_z \, dz
&\leq \frac{1}{6} \| \bar{\theta}_z \|_{L^2(0,2\pi)}^2 + C \int_0^{2\pi} (\overline{\theta' w_2})^2 \, dz \notag \\
&\leq \frac{1}{6} \| \bar{\theta}_z \|_{L^2(0,2\pi)}^2 + C \| A^{-1/2} \theta' \|_2^2 \left( \sup_{z \in [0,2\pi]} \| A^{1/2} w_2 \|_{L^2([0,2\pi]^2)}^2 \right).
\end{align}
By the embedding $W^{\frac{1}{3},4}(0,2\pi) \hookrightarrow L^{\infty}(0,2\pi)$, we have
\begin{align} \label{fifth2}
\sup_{z \in [0,2\pi]} \| A^{1/2} w_2 \|_{L^2([0,2\pi]^2)}^2
&\leq  \int_{[0,2\pi]^2} \left( \sup_{z \in [0,2\pi]} |A^{1/2} w_2| \right)^2 dx \, dy \notag \\
&\leq C \int_{[0,2\pi]^2} \| (I - \partial_{zz})^{1/6} A^{1/2} w_2 \|_{L^4(0,2\pi)}^2 dx \, dy \notag \\
&\leq C \| (I - \partial_{zz})^{1/6} A^{1/2} w_2 \|_4^2 
\leq C \| \theta'_2 \|_4^2,
\end{align}
by using the relation between the Fourier coefficients of \( w_2 \) and \( \theta'_2 \) as in (\ref{F-w}),  
and the Fourier multiplier result: Proposition \ref{Fmre}.

Therefore, from (\ref{fifth1}) and (\ref{fifth2}), we obtain
\begin{align} \label{fifth3}
\int_0^{2\pi} (\overline{\theta' w_2}) \, \bar{\theta}_z \, dz
\leq \frac{1}{6} \| \bar{\theta}_z \|_{L^2(0,2\pi)}^2 + C \| A^{-1/2} \theta' \|_2^2 \, \| \theta'_2 \|_4^2.
\end{align}

Finally, we estimate the last term on the right-hand side of (\ref{ntd1}).  
By Hölder's inequality, we have
\begin{align} \label{sixth1}
\int_0^{2\pi} (\overline{\theta'_1 w}) \, \bar{\theta}_z \, dz
&\leq C \int_0^{2\pi} \| \theta'_1 \|_{L^3([0,2\pi]^2)} \| w \|_{L^{3/2}([0,2\pi]^2)} |\bar{\theta}_z| \, dz \notag \\
&\leq C \| \theta'_1 \|_3 \left( \int_0^{2\pi} \| w \|_{L^{3/2}([0,2\pi]^2)}^6 \, dz \right)^{1/6} \| \bar{\theta}_z \|_{L^2(0,2\pi)}.
\end{align}

Thanks to Minkowski's integral inequality, we have
\begin{align} \label{sixth2}
\left( \int_0^{2\pi} \| w \|_{L^{3/2}([0,2\pi]^2)}^6 \, dz \right)^{1/6}
&\leq \left( \int_{[0,2\pi]^2} \| w \|_{L^6(0,2\pi)}^{3/2} \, dx \, dy \right)^{2/3} \notag \\
&\leq C \left( \int_{[0,2\pi]^2} \| (I - \partial_{zz})^{1/6} w \|_{L^2(0,2\pi)}^{3/2} \, dx \, dy \right)^{2/3} \notag \\
&\leq C \left( \int_{[0,2\pi]^2} \| A^{-1/2} \theta' \|_{L^2(0,2\pi)}^{3/2} \, dx \, dy \right)^{2/3} \notag \\
&\leq C \| A^{-1/2} \theta' \|_2,
\end{align}
where we have used the embedding \( H^{1/3}(0,2\pi) \hookrightarrow L^6(0,2\pi) \) and the relation (\ref{F-w}).

Therefore, from (\ref{sixth1}) and (\ref{sixth2}), we obtain
\begin{align} \label{sixth3}
\int_0^{2\pi} (\overline{\theta'_1 w}) \, \bar{\theta}_z \, dz 
\leq \frac{1}{6} \| \bar{\theta}_z \|_{L^2(0,2\pi)}^2 + C \| A^{-1/2} \theta' \|_2^2 \, \| \theta'_1 \|_3^2.
\end{align}

Substituting (\ref{first3}), (\ref{Second3}), (\ref{ntd3}), (\ref{third5}), (\ref{fifth3}), and (\ref{sixth3}) into (\ref{ntd1}) yields
\begin{align} \label{ntd5}
&\frac{d}{dt} \| A^{-1/2} \theta' \|_2^2 + \| \bar{\theta}_z \|_2^2  \notag\\
&\leq C \| A^{-1/2} \theta' \|_2^2 \left( \| \theta'_1 \|_5 + \| \theta'_2 \|_6 + \| \theta'_1 \|_3^2 + \| \theta'_2 \|_4^2 + \| \partial_z \bar{\theta}_1 \|_{L^2(0,2\pi)} \right),
\end{align}
for all \( t \in [0, T] \).

Applying Grönwall’s inequality to (\ref{ntd5}), we obtain
\begin{align} \label{ntd6}
\| A^{-1/2} \theta'(t) \|_2^2 
\leq e^{C \int_0^t \left( \| \theta'_1 \|_5 + \| \theta'_2 \|_6 + \| \theta'_1 \|_3^2 + \| \theta'_2 \|_4^2 + \| \partial_z \bar{\theta}_1 \|_{L^2(0,2\pi)} \right) ds} \, \| A^{-1/2} \theta'(0) \|_2^2,
\end{align}
for all \( t \in [0, T] \).

Since \( \theta'(0) = \theta'_1(0) - \theta'_2(0) = 0 \), it follows from (\ref{ntd6}) that \( \theta' = 0 \) on \( [0, T] \).  
Moreover, due to (\ref{F-uv}) and (\ref{F-w}), we also obtain \( \mathbf{u} = 0 \) and \( w = 0 \) on \( [0, T] \).

Finally, since \( \theta' = 0 \) on \( [0, T] \), it follows from (\ref{ntd5}) that \( \bar{\theta}_z = 0 \) on \( [0, T] \).  
Given that \( \bar{\theta} \in \dot{H}^1(0,2\pi) \), we have \( \int_0^{2\pi} \bar{\theta} \, dz = 0 \).  
Therefore, \( \bar{\theta} = 0 \) on \( [0, T] \).

This completes the proof of the uniqueness of weak solutions to the system (\ref{iR-1})--(\ref{iR-4}) stated in Theorem \ref{wellp}.

\vspace{0.1 in}

\section{Global Well-Posedness of Strong Solutions}

In this section, we prove Theorem \ref{wellp2}, which establishes the global well-posedness of strong solutions to the system (\ref{iRR-1})--(\ref{iRR-4}) with initial data $\theta'_0 \in L^6(\Omega)$ and $\nabla_h \theta'_0 \in L^3(\Omega)$. The regularity of strong solutions enables us to demonstrate not only uniqueness but also continuous dependence on the initial data. This contrasts with the weak solutions established in the previous section, where we proved existence and uniqueness but not continuous dependence. The reason is that the uniqueness argument for weak solutions relies on estimates in a dual space, which cannot be directly adapted to establish continuous dependence on initial data.

Moreover, the strong solutions developed here will play a crucial role in justifying the vanishing diffusivity limit in the next section.

The uniqueness of strong solutions follows directly from the uniqueness of weak solutions. Therefore, it remains to show the global existence of strong solutions and their continuous dependence on the initial data.

\subsection{Global Existence of Strong Solutions}

The strategy for proving the existence of strong solutions to   system (\ref{iRR-1})--(\ref{iRR-4})    is similar to the approach used previously to establish weak solutions. We consider the system (\ref{fRss-1})--(\ref{fRss-4}) below as a regularization of the system (\ref{iRR-1})--(\ref{iRR-4}). Note that the system (\ref{fRss-1})--(\ref{fRss-4}) includes horizontal thermal diffusion, which is absent in the system (\ref{iRR-1})--(\ref{iRR-4}).

Given initial data $\theta'_0 \in L^6(\Omega)$ with $\nabla_h \theta'_0 \in L^3(\Omega)$ and \( \bar{\theta'_{0}} = 0 \), there exists a sequence of functions $\theta'_{0,\varepsilon} \in H^2(\Omega)$ such that $\theta'_{0,\varepsilon} \to \theta'_0$ in $L^6(\Omega)$, $\nabla_h \theta'_{0,\varepsilon} \to \nabla_h \theta'_0$ in $L^3(\Omega)$, and \( \bar{\theta'_{0,\varepsilon}} = 0 \). 

In \cite{CGT-3}, we established the global existence of strong solutions to system (\ref{fRss-1})--(\ref{fRss-4}) with $H^1$ initial data using the Galerkin method. By extending this approach, one can show the global existence of more regular solutions with $H^2$ initial data. Consequently, the system (\ref{fRss-1})--(\ref{fRss-4}) admits a sequence of global regular solutions $(\theta'_{\varepsilon}, \bar{\theta}_{\varepsilon}, \mathbf{u}_{\varepsilon}, w_{\varepsilon})$ with initial data $\theta'_{\varepsilon}(0) = \theta'_{0,\varepsilon} \in H^2(\Omega)$:
\begin{align}
& \partial_z \psi_{\varepsilon} = \theta'_{\varepsilon} + \Delta_h w_{\varepsilon}, && \text{in } L^{\infty}(0,T; H^2(\Omega)), \label{fRss-1} \\
& -\partial_z w_{\varepsilon} = \Delta_h \omega_{\varepsilon}, && \text{in } L^{\infty}(0,T; H^1(\Omega)), \label{fRss-2} \\
& \partial_t \theta'_{\varepsilon} + \mathbf{u}_{\varepsilon} \cdot \nabla_h \theta'_{\varepsilon} + w_{\varepsilon} \partial_z \bar{\theta}_{\varepsilon} = \varepsilon^2 \Delta_h \theta'_{\varepsilon}, && \text{in } L^2(0,T; H^1(\Omega)), \label{fRss-3} \\
& \partial_z(\overline{\theta'_{\varepsilon} w_{\varepsilon}}) = \partial_{zz} \bar{\theta}_{\varepsilon}, && \text{in } L^{\infty}(0,T; H^1(0,2\pi)), \label{fRss-4}
\end{align}
with $\nabla_h \cdot \mathbf{u}_{\varepsilon} = 0$, for any $T>0$. The regular solutions satisfy:
\begin{align*}
& \theta'_{\varepsilon} \in L^{\infty}(0,T; H^2(\Omega)) \cap C([0,T]; H^1(\Omega)), \quad \Delta_h \theta'_{\varepsilon}  \in L^2(0,T; H^1(\Omega)),\\
& \bar{\theta}_{\varepsilon} \in L^{\infty}(0,T; \dot{H}^3(0,2\pi)), \quad \Delta_h \mathbf{u}_{\varepsilon}, \Delta_h w_{\varepsilon} \in L^{\infty}(0,T; H^2(\Omega)).
\end{align*}

We claim that as $\varepsilon \to 0$, the sequence $(\theta'_{\varepsilon}, \bar{\theta}_{\varepsilon}, \mathbf{u}_{\varepsilon}, w_{\varepsilon})$ converges to a function $(\theta', \bar{\theta}, \mathbf{u}, w)$, which is a strong solution of the system (\ref{iRR-1})--(\ref{iRR-4}). A uniform bound for the $L^6$-norm of $\theta'_{\varepsilon}$ has already been established in (\ref{Lptheta5}). It remains to derive a uniform bound for $\|\nabla_h \theta'_{\varepsilon}\|_3$ on $[0,T]$, independent of $\varepsilon$.

\subsubsection{Estimate for $\|\nabla_h \theta'_{\varepsilon}\|_3$}

We differentiate equation (\ref{fRss-3}) with respect to $x$, multiply the result by $|\partial_x \theta'_{\varepsilon}| \partial_x \theta'_{\varepsilon}$, and integrate over $\Omega$. This yields:
\begin{align}     \label{W1p2}
&\frac{1}{3} \frac{d}{dt}\|\partial_x \theta'_{\varepsilon}\|_3^3 
+ \int_{\Omega}   ( \partial_x \mathbf{u}_{\varepsilon} \cdot \nabla_h \theta'_{\varepsilon})  |\partial_x \theta'_{\varepsilon}| \partial_x \theta'_{\varepsilon} \, dx\,dy\,dz  
+  \int_{\Omega}   (\partial_x w_{\varepsilon}) (\partial_z \bar{\theta}_{\varepsilon}) |\partial_x \theta'_{\varepsilon}| \partial_x \theta'_{\varepsilon} \, dx\,dy\,dz     \notag\\
&\hspace{0.2 in} +  2  \varepsilon^2   \int_{\Omega} (|\partial_{xx} \theta'_{\varepsilon}|^2 + |\partial_{xy} \theta'_{\varepsilon}|^2)|\partial_x \theta'_{\varepsilon}| \, dx\,dy\,dz = 0,
\end{align}
for all $t \geq 0$.

We now estimate the nonlinear terms in (\ref{W1p2}). Observe that
\begin{align}    \label{W1p3}
\int_{\Omega}  \left| ( \partial_x \mathbf{u}_{\varepsilon} \cdot \nabla_h \theta'_{\varepsilon}) (\partial_x \theta'_{\varepsilon})^2 \right| \, dx\,dy\,dz  
\leq \|\partial_x \mathbf{u}_{\varepsilon} \|_{\infty} \|\nabla_h \theta'_{\varepsilon}\|_3^3.
\end{align}

We aim to show that $\|\partial_x \mathbf{u}_{\varepsilon}\|_{\infty} \leq C \|\theta'_{\varepsilon}\|_6$. This will provide a uniform bound for $\|\partial_x \mathbf{u}_{\varepsilon}\|_{\infty}$, since $\|\theta'_{\varepsilon}\|_6$ is uniformly bounded by (\ref{Lptheta5}).

Let $A = -\Delta_h$. Recall the 2D Sobolev embedding $W^{\frac{2}{5},6}([0,2\pi]^2) \hookrightarrow L^{\infty}([0,2\pi]^2)$ and the 1D embedding $W^{\frac{1}{5},6}(0,2\pi) \hookrightarrow L^{\infty}(0,2\pi)$. Therefore, for any sufficiently regular function $f$, we have:
\begin{align}  \label{W1p4be}
\|f\|_{\infty} &\leq \sup_{z\in [0,2\pi]}  \left(C\|f(z)\|_{W^{\frac{2}{5},6}( [0,2\pi]^2)}\right)   \leq   C\left(\int_{ [0,2\pi]^2} \sup_{z\in [0,2\pi]} |A^{1/5} f|^6 \, dx\,dy \right)^{1/6}  \notag\\
&\leq  C\left(\int_{ [0,2\pi]^2} \|A^{1/5} f\|^6_{W^{\frac{1}{5},6}(0,2\pi)} \, dx\,dy \right)^{1/6}  
= C\|  (I-\partial_{zz})^{1/10} A^{1/5}  f\|_6.
\end{align}

Substituting $f = \partial_x \mathbf{u}_{\varepsilon}$ into (\ref{W1p4be}) gives:
\begin{align}  \label{mult-6}
\|\partial_x \mathbf{u}_{\varepsilon}\|_{\infty}   \leq C\|  (I-\partial_{zz})^{1/10}A^{1/5}  \partial_x \mathbf{u}_{\varepsilon}\|_6
\leq C\|\theta'_{\varepsilon}\|_6,
\end{align}
where the last inequality follows from (\ref{Four-uv}) and Proposition \ref{Fmre}.

Combining (\ref{W1p3}) and (\ref{mult-6}), we obtain:
\begin{align}    \label{mult-7}
\int_{\Omega}  \left| ( \partial_x \mathbf{u}_{\varepsilon} \cdot \nabla_h \theta'_{\varepsilon}) (\partial_x \theta'_{\varepsilon})^2 \right| \, dx\,dy\,dz 
\leq C \|\theta'_{\varepsilon}\|_6 \|\nabla_h \theta'_{\varepsilon}\|_3^3.
\end{align}

Next, we estimate the remaining nonlinear term in (\ref{W1p2}). Using (\ref{L3theta1'}) and applying Hölder's inequality, we obtain:
\begin{align}  \label{W1p5}
\int_{\Omega}  \left| (\partial_x w_{\varepsilon}) (\partial_z \bar{\theta}_{\varepsilon}) (\partial_x \theta'_{\varepsilon})^2 \right| \, dx\,dy\,dz
\leq C \|\partial_x w_{\varepsilon}\|_{\infty} \|w_{\varepsilon}\|_{\infty} \|\partial_x \theta'_{\varepsilon}\|_3^2 \|\theta'_{\varepsilon}\|_3.
\end{align}

Analogous to (\ref{mult-6}), we can derive:
\begin{align}    \label{W1p5'}
\|\partial_x w_{\varepsilon}\|_{\infty} + \|w_{\varepsilon}\|_{\infty} \leq C\|\theta'_{\varepsilon}\|_6.
\end{align}

Combining (\ref{W1p5}) and (\ref{W1p5'}), we obtain:
\begin{align}  \label{W1p5''}
\int_{\Omega}  \left| (\partial_x w_{\varepsilon}) (\partial_z \bar{\theta}_{\varepsilon}) (\partial_x \theta'_{\varepsilon})^2 \right| \, dx\,dy\,dz 
\leq C \|\theta'_{\varepsilon}\|_6^2 \|\partial_x \theta'_{\varepsilon}\|_3^2 \|\theta'_{\varepsilon}\|_3.
\end{align}

From (\ref{W1p2}), (\ref{mult-7}), and (\ref{W1p5''}), we conclude:
\begin{align}   \label{W1p8b}
\frac{d}{dt}\|\partial_x \theta'_{\varepsilon}\|_3^3 \leq C (\|\theta'_{\varepsilon}\|_6^2 + 1) \|\nabla_h \theta'_{\varepsilon}\|_3^3,
\quad \text{for all } t \geq 0.
\end{align}

A similar estimate can be derived for $\frac{d}{dt}\|\partial_y \theta'_{\varepsilon}\|_3^3$. Then, applying Gr\"onwall's inequality and using (\ref{Lptheta5}), we obtain:
\begin{align}    \label{W1p8}
\|\nabla_h \theta'_{\varepsilon}(t)\|_3^3
&\leq \|\nabla_h \theta'_0\|_3^3 \exp \left( \int_0^t C(\|\theta'_{\varepsilon}\|_6^2 + 1) \, ds \right) \notag\\
&\leq \|\nabla_h \theta'_0\|_3^3 \exp \left(C \int_0^t \left(\|\theta'_0\|_6^2  
\exp\left( C \|\theta'_0\|_3^2 \int_0^s e^{C\|\theta'_0\|_2^2 \tau} \, d\tau \right) + 1\right) ds \right),
\end{align}
for all $t \geq 0$.

This shows that $\|\nabla_h \theta'_{\varepsilon}(t)\|_3$ is bounded by a function that grows triple-exponentially in time, and this bound is independent of $\varepsilon$.

Using the uniform bound (\ref{Lptheta5}) for $\|\theta'_{\varepsilon}(t)\|_6$ and the uniform bound (\ref{W1p8}) for $\|\nabla_h \theta'_{\varepsilon}(t)\|_3$, and applying a standard argument similar to that used for weak solutions in Subsection \ref{passage}, we can pass to the limit as $\varepsilon \rightarrow 0$ and establish the existence of a global strong solution to the system (\ref{iRR-1})--(\ref{iRR-4}).

\vspace{0.1 in}

\subsection{Continuous Dependence on Initial Data}

Given a sequence of initial data $\theta_{0,n}' \to \theta_0'$ in $L^2(\Omega)$ such that $\theta'_{0,n}, \theta'_0 \in L^6(\Omega)$, $\nabla_h \theta'_{0,n}, \nabla_h \theta'_0 \in L^3(\Omega)$, and $\bar{\theta'_{0,n}} = \bar{\theta'_0} = 0$, consider the corresponding strong solutions $\theta'_n$ and $\theta'$ with $\theta'_n(0) = \theta'_{0,n}$ and $\theta'(0) = \theta'_0$. We aim to show that $\theta'_n \to \theta'$ in $L^{\infty}(0,T; L^2(\Omega))$.

Let $\tilde \theta'_n = \theta'_n - \theta'$, $\tilde{\bar{\theta}}_n = \bar{\theta}_n - \bar{\theta}$, $\tilde{\mathbf{u}}_n = \mathbf{u}_n - \mathbf{u}$, $\tilde{w}_n = w_n - w$, $\tilde{\omega}_n = \omega_n - \omega$, and $\tilde{\psi}_n = \psi_n - \psi$. Then the following system holds:
\begin{align}
& \partial_z \tilde \psi_n = \tilde \theta'_n + \Delta_h \tilde w_n, && \text{in } L^{\infty}(0,T; H^1_h(\Omega)); \label{inun-11} \\
& -\partial_z \tilde w_n = \Delta_h \tilde \omega_n, && \text{in } L^{\infty}(0,T; L^3(\Omega)); \label{inun-22} \\
& \partial_t \tilde \theta'_n + \tilde{\mathbf{u}}_n \cdot \nabla_h \theta' + \mathbf{u}_n \cdot \nabla_h \tilde \theta'_n + \tilde w_n \partial_z \bar{\theta} + w_n \partial_z \tilde{\bar{\theta}}_n = 0, && \text{in } L^2(\Omega \times (0,T)); \label{inun-33} \\
& \partial_z(\overline{\tilde \theta'_n w_n}) + \partial_z(\overline{\theta' \tilde w_n}) = \partial_{zz} \tilde{\bar{\theta}}_n, && \text{in } L^2(0,T; H^{-1}(0,2\pi)), \label{inun-44}
\end{align}
with $\nabla_h \cdot \tilde{\mathbf{u}}_n = 0$.

Taking the $L^2$ inner product of (\ref{inun-33}) with $\tilde \theta'_n$ yields:
\begin{align} \label{ntd11}
\frac{1}{2} \frac{d}{dt} \|\tilde \theta'_n\|_2^2 
= & -\int_{\Omega} (\tilde{\mathbf{u}}_n \cdot \nabla_h \theta') \tilde \theta'_n \, dx\,dy\,dz 
- \int_{\Omega} \tilde w_n (\partial_z \bar{\theta}) \tilde \theta'_n \, dx\,dy\,dz \notag \\
& - \int_{\Omega} w_n (\partial_z \tilde{\bar{\theta}}_n) \tilde \theta'_n \, dx\,dy\,dz.
\end{align}

Applying Hölder's inequality, we estimate:
\begin{align} \label{ntd2'}
\int_{\Omega} \left| (\tilde{\mathbf{u}}_n \cdot \nabla_h \theta') \tilde \theta'_n \right| \, dx\,dy\,dz 
\leq  \|\tilde{\mathbf{u}}_n\|_6 \|\nabla_h \theta'\|_3 \|\tilde \theta'_n\|_2.
\end{align}

Let $A = -\Delta_h$. We aim to show that $\|\tilde{\mathbf{u}}_n\|_6$ is bounded by $\|\tilde \theta'_n\|_2$. Using Sobolev embeddings in the horizontal and vertical directions, the 1D embedding $H^{1/3}(0,2\pi) \hookrightarrow L^6(0,2\pi)$, the 2D embedding $H^{2/3}([0,2\pi]^2) \hookrightarrow L^6([0,2\pi]^2)$, and Minkowski’s integral inequality, we obtain:
\begin{align} \label{z6}
\|\tilde{\mathbf{u}}_n\|_6 
&= \left( \int_{[0,2\pi]^2} \|\tilde{\mathbf{u}}_n\|_{L^6(0,2\pi)}^6 \, dx\,dy \right)^{1/6}
\leq C \left( \int_{[0,2\pi]^2} \|\tilde{\mathbf{u}}_n\|_{H^{1/3}(0,2\pi)}^6 \, dx\,dy \right)^{1/6} \notag \\
&= C \left( \int_{[0,2\pi]^2} \|(I - \partial_{zz})^{1/6} \tilde{\mathbf{u}}_n\|_{L^2(0,2\pi)}^6 \, dx\,dy \right)^{1/6} \notag \\
&\leq C \left( \int_0^{2\pi} \|(I - \partial_{zz})^{1/6} \tilde{\mathbf{u}}_n\|_{L^6([0,2\pi]^2)}^2 \, dz \right)^{1/2}   \notag\\
&\leq C \|A^{1/3}(I - \partial_{zz})^{1/6} \tilde{\mathbf{u}}_n\|_2 \leq C \|\tilde \theta'_n\|_2,
\end{align}
where the last inequality follows from (\ref{F-uv}).

Combining (\ref{ntd2'}) and (\ref{z6}), we obtain:
\begin{align} \label{ntd22}
\int_{\Omega} \left| (\tilde{\mathbf{u}}_n \cdot \nabla_h \theta') \tilde \theta'_n \right| \, dx\,dy\,dz 
\leq C \|\tilde \theta'_n\|_2^2 \|\nabla_h \theta'\|_3.
\end{align}

Next, we estimate the second term on the right-hand side of (\ref{ntd11}). Using the 1D embedding $H^s(0,2\pi) \hookrightarrow L^{\infty}(0,2\pi)$ for $s > 1/2$, we have:
\begin{align} \label{ntd33}
\int_{\Omega} \left| \tilde w_n \frac{\partial \bar{\theta}}{\partial z} \tilde \theta'_n \right| \, dx\,dy\,dz 
&\leq \|\tilde \theta'_n\|_2 \left( \int_0^{2\pi} \left| \frac{\partial \bar{\theta}}{\partial z} \right|^2 \, dz \right)^{1/2}
\left( \int_{[0,2\pi]^2} \sup_{z \in [0,2\pi]} |\tilde w_n|^2 \, dx\,dy \right)^{1/2} \notag \\
&\leq C \|\tilde \theta'_n\|_2 \left( \int_0^{2\pi} \left| \frac{\partial \bar{\theta}}{\partial z} \right|^2 \, dz \right)^{1/2}
\|(I - \partial_{zz})^{1/3} \tilde w_n\|_2 \notag \\
&\leq C \|\tilde \theta'_n\|_2^2 \left( \int_0^{2\pi} \left| \frac{\partial \bar{\theta}}{\partial z} \right|^2 \, dz \right)^{1/2},
\end{align}
where the last inequality follows from (\ref{F-w}).

Now we consider the third term on the right-hand side of (\ref{ntd11}). From (\ref{inun-44}), we know that
\[
\frac{\partial(\overline{\tilde \theta'_n w_n})}{\partial z} = \frac{\partial^2 \tilde{\bar \theta}_n}{\partial z^2} - \frac{\partial(\overline{\theta' \tilde w_n})}{\partial z} \quad \text{in } L^2(0,T; H^{-1}(0,2\pi)),
\]
and thus,
\begin{align} \label{inu-8}
- \frac{1}{4\pi^2} \int_{\Omega} w_n (\partial_z \tilde{\bar{\theta}}_n) \tilde \theta'_n \, dx\,dy\,dz
&= -\int_0^{2\pi} \overline{w_n \tilde \theta'_n} \, \partial_z \tilde{\bar{\theta}}_n \, dz
= \left\langle \frac{\partial (\overline{w_n \tilde \theta'_n})}{\partial z}, \tilde{\bar{\theta}}_n \right\rangle \notag \\
&= \left\langle \frac{\partial^2 \tilde{\bar \theta}_n}{\partial z^2} - \frac{\partial(\overline{\theta' \tilde w_n})}{\partial z}, \tilde{\bar{\theta}}_n \right\rangle \notag \\
&= -\int_0^{2\pi} |\partial_z \tilde{\bar{\theta}}_n|^2 \, dz + \int_0^{2\pi} \overline{\theta' \tilde w_n} \, \partial_z \tilde{\bar{\theta}}_n \, dz \notag \\
&\leq -\frac{1}{2} \int_0^{2\pi} |\partial_z \tilde{\bar{\theta}}_n|^2 \, dz + \frac{1}{2} \int_0^{2\pi} |\overline{\theta' \tilde w_n}|^2 \, dz.
\end{align}

Similar to the estimate in (\ref{ntd33}), we deduce:
\begin{align} \label{inu-9}
\int_0^{2\pi} |\overline{\theta' \tilde w_n}|^2 \, dz 
&\leq C \int_0^{2\pi} \left( \int_{[0,2\pi]^2} |\theta'|^2 \, dx\,dy \right) \left( \int_{[0,2\pi]^2} |\tilde w_n|^2 \, dx\,dy \right) dz \notag \\
&\leq C \left( \int_{[0,2\pi]^2} \sup_{z \in [0,2\pi]} |\tilde w_n|^2 \, dx\,dy \right) \|\theta'\|_2^2 \notag \\
&\leq C \|(I - \partial_{zz})^{1/3} \tilde w_n\|_2^2 \|\theta'\|_2^2 \leq C \|\tilde \theta'_n\|_2^2 \|\theta'\|_2^2.
\end{align}

Combining (\ref{inu-8}) and (\ref{inu-9}) gives:
\begin{align} \label{ntd44}
- \int_{\Omega} w_n (\partial_z \tilde{\bar{\theta}}_n) \tilde \theta'_n \, dx\,dy\,dz 
\leq -2\pi^2 \int_0^{2\pi} |\partial_z \tilde{\bar{\theta}}_n|^2 \, dz + C \|\tilde \theta'_n\|_2^2 \|\theta'\|_2^2.
\end{align}

Substituting (\ref{ntd22}), (\ref{ntd33}), and (\ref{ntd44}) into (\ref{ntd11}) yields:
\begin{align} \label{ntd55}
\frac{d}{dt} \|\tilde \theta'_n\|_2^2 
\leq C \|\tilde \theta'_n\|_2^2 \left( \|\nabla_h \theta'\|_3 + \left( \int_0^{2\pi} \left| \frac{\partial \bar{\theta}}{\partial z} \right|^2 dz \right)^{1/2} + \|\theta'\|_2^2 \right),
\end{align}
for all \( t \in [0,T] \).

Applying Grönwall’s inequality to (\ref{ntd55}), we obtain:
\begin{align} \label{ntd66}
\|\tilde \theta'_n(t)\|_2^2 
\leq \exp\left(C \int_0^t \left[ \|\nabla_h \theta'\|_3 + \left( \int_0^{2\pi} \left| \frac{\partial \bar{\theta}}{\partial z} \right|^2 dz \right)^{1/2} + \|\theta'\|_2^2 \right] ds \right) \|\tilde \theta'_n(0)\|_2^2.
\end{align}

Since \( \|\tilde \theta'_n(0)\|_2 = \|\theta'_{0,n} - \theta'_0\|_2 \to 0 \) as \( n \to \infty \), it follows that:
\begin{align} \label{ntd66'}
\sup_{t \in [0,T]} \|\tilde \theta'_n(t)\|_2 \to 0 \quad \text{as } n \to \infty, \quad \text{that is, } \theta'_n \to \theta' \text{ in } L^{\infty}(0,T; L^2(\Omega)).
\end{align}

Due to the linear relations between Fourier coefficients in (\ref{F-uv}) and (\ref{F-w}), we have
$\|\Delta_h \tilde{\mathbf{u}}_n\|_2^2 + \|\Delta_h \tilde{w}_n\|_2^2 \leq C \|\tilde \theta'_n\|_2^2$,
and along with (\ref{ntd66'}), we obtain \( \mathbf{u}_n \to \mathbf{u} \) and \( w_n \to w \) in \( L^{\infty}(0,T; H^2_h(\Omega)) \).

To establish continuous dependence for the horizontal average temperature \( \bar{\theta} \), we deduce from (\ref{inun-44}) that:
\begin{align} \label{ntd77}
\int_0^{2\pi} |\partial_z \tilde{\bar{\theta}}_n|^2 dz 
&= \int_0^{2\pi} \left[ \overline{\tilde \theta'_n w_n} + \overline{\theta' \tilde w_n} \right] \partial_z \tilde{\bar{\theta}}_n \, dz \notag \\
&\leq \frac{1}{2} \int_0^{2\pi} |\partial_z \tilde{\bar{\theta}}_n|^2 dz + \frac{1}{2} \int_0^{2\pi} \left|\overline{\tilde \theta'_n w_n}\right|^2 dz + \frac{1}{2} \int_0^{2\pi} \left|\overline{\theta' \tilde w_n}\right|^2 dz.
\end{align}

Therefore,
\begin{align} \label{ntd88}
\int_0^{2\pi} |\partial_z \tilde{\bar{\theta}}_n|^2 dz 
&\leq \int_0^{2\pi} \left|\overline{\tilde \theta'_n w_n}\right|^2 dz + \int_0^{2\pi} \left|\overline{\theta' \tilde w_n}\right|^2 dz \notag \\
&\leq C \|\tilde \theta'_n\|_2^2 \sup_{z \in [0,2\pi]} \|w_n\|_{L^2([0,2\pi]^2)}^2 + C \|\theta'\|_2^2 \sup_{z \in [0,2\pi]} \|\tilde w_n\|_{L^2([0,2\pi]^2)}^2 \notag \\
&\leq C \|\tilde \theta'_n\|_2^2 \|(I - \partial_{zz})^{1/3} w_n\|_2^2 
+ C \|\theta'\|_2^2 \|(I - \partial_{zz})^{1/3} \tilde w_n\|_2^2 \notag \\
&\leq C \|\tilde \theta'_n\|_2^2 \left( \|\theta'_n\|_2^2 + \|\theta'\|_2^2 \right),
\end{align}
for all \( t \in [0,T] \). 

Since \( \sup_{t \in [0,T]} \|\tilde \theta'_n(t)\|_2 \to 0 \) as \( n \to \infty \) by (\ref{ntd66'}), it follows that
\[
\sup_{t \in [0,T]} \int_0^{2\pi} |\partial_z \tilde{\bar{\theta}}_n|^2 \, dz \to 0 \quad \text{as } n \to \infty.
\]
This implies that
\[
\bar{\theta}_n \to \bar{\theta} \quad \text{in } L^{\infty}(0,T; \dot{H}^1(0,2\pi)).
\]

This completes the proof of the continuous dependence on initial data for strong solutions.

\vspace{0.1in}

\section{Vanishing Diffusivity Limit}

This section is devoted to proving Theorem \ref{thm-conv}, which concerns the vanishing diffusivity limit of the system (\ref{fR-1})--(\ref{fR-4}). Specifically, we consider the weak solution $(\theta'_{\varepsilon}, \bar{\theta}_{\varepsilon}, \mathbf{u}_{\varepsilon}, w_{\varepsilon})$ of system (\ref{fR-1})--(\ref{fR-4}) with initial data $\theta'_{\varepsilon}(0) \in L^2(\Omega)$, and the strong solution $(\theta', \bar{\theta}, \mathbf{u}, w)$ of system (\ref{iRR-1})--(\ref{iRR-4}) with initial data $\theta'_0 \in L^6(\Omega)$ and $\nabla_h \theta'_0 \in L^3(\Omega)$.

Theorem \ref{thm-conv} asserts that $(\theta'_{\varepsilon}, \bar{\theta}_{\varepsilon}, \mathbf{u}_{\varepsilon}, w_{\varepsilon})$ converges to $(\theta', \bar{\theta}, \mathbf{u}, w)$ in $L^{\infty}(0,T; L^2(\Omega) \times \dot{H}^1(0,2\pi) \times (H^2_h(\Omega))^3)$
as $\varepsilon \to 0$, provided that $\theta'_{\varepsilon}(0) \to \theta'_0$ in $L^2(\Omega)$.

Let us define the differences:
\[
\Theta'_{\varepsilon} = \theta'_{\varepsilon} - \theta', \;\; \mathbf{U}_{\varepsilon} = \mathbf{u}_{\varepsilon} - \mathbf{u}, \;\; W_{\varepsilon} = w_{\varepsilon} - w, \;\; \bar{\Theta}_{\varepsilon} = \bar{\theta}_{\varepsilon} - \bar{\theta},\;\;
\Psi_{\varepsilon} = \psi_{\varepsilon} - \psi, \;\; \Omega_{\varepsilon} = \omega_{\varepsilon} - \omega.
\]

Subtracting systems (\ref{iRR-1})--(\ref{iRR-4}) from (\ref{fR-1})--(\ref{fR-4}) yields:
\begin{align}
& \partial_z \Psi_{\varepsilon} = \Theta'_{\varepsilon} + \Delta_h W_{\varepsilon}, && \text{in } L^2(0,T; H^1_h(\Omega)), \label{df-1} \\
& -\partial_z W_{\varepsilon} = \Delta_h \Omega_{\varepsilon}, && \text{in } L^2(\Omega \times (0,T)), \label{df-2} \\
& \partial_t \Theta'_{\varepsilon} + \mathbf{U}_{\varepsilon} \cdot \nabla_h \theta' + \mathbf{u}_{\varepsilon} \cdot \nabla_h \Theta'_{\varepsilon} + W_{\varepsilon} \partial_z \bar{\theta} + w_{\varepsilon} \partial_z \bar{\Theta}_{\varepsilon} \notag \\
& \hspace{1cm} = \varepsilon^2 \Delta_h \Theta'_{\varepsilon} + \varepsilon^2 \Delta_h \theta', && \text{in } L^2(0,T; (H^1_h(\Omega))'), \label{df-3} \\
& \partial_z(\overline{\theta' W_{\varepsilon}}) + \partial_z(\overline{\Theta'_{\varepsilon} w_{\varepsilon}}) = \partial_{zz} \bar{\Theta}_{\varepsilon}, && \text{in } L^2(0,T; H^{-1}(0,2\pi)), \label{df-4}
\end{align}
with $\nabla_h \cdot \mathbf{u}_{\varepsilon} = 0$ and $\nabla_h \cdot \mathbf{U}_{\varepsilon} = 0$.

Taking the duality pairing of (\ref{df-3}) with $\Theta'_{\varepsilon}$ in $(H^1_h(\Omega))' \times H^1_h(\Omega)$ gives:
\begin{align*}
\frac{1}{2} \frac{d}{dt} \|\Theta'_{\varepsilon}\|_2^2 + \varepsilon^2 \|\nabla_h \Theta'_{\varepsilon}\|_2^2 
&+ \int_{\Omega} (\mathbf{U}_{\varepsilon} \cdot \nabla_h \theta') \Theta'_{\varepsilon} \, dx\,dy\,dz 
+ \int_{\Omega} W_{\varepsilon} \partial_z \bar{\theta} \, \Theta'_{\varepsilon} \, dx\,dy\,dz \\
&+ \int_{\Omega} w_{\varepsilon} \partial_z \bar{\Theta}_{\varepsilon} \, \Theta'_{\varepsilon} \, dx\,dy\,dz 
\leq \frac{\varepsilon^2}{2} \|\nabla_h \theta'\|_2^2 + \frac{\varepsilon^2}{2} \|\nabla_h \Theta'_{\varepsilon}\|_2^2.
\end{align*}

This leads to:
\begin{align} \label{con1}
\frac{d}{dt} \|\Theta'_{\varepsilon}\|_2^2 &+ \varepsilon^2 \|\nabla_h \Theta'_{\varepsilon}\|_2^2 
+ 2 \int_{\Omega} (\mathbf{U}_{\varepsilon} \cdot \nabla_h \theta') \Theta'_{\varepsilon} \, dx\,dy\,dz \notag \\
&+ 2 \int_{\Omega} W_{\varepsilon} \partial_z \bar{\theta} \, \Theta'_{\varepsilon} \, dx\,dy\,dz 
+ 2 \int_{\Omega} w_{\varepsilon} \partial_z \bar{\Theta}_{\varepsilon} \, \Theta'_{\varepsilon} \, dx\,dy\,dz 
\leq \varepsilon^2 \|\nabla_h \theta'\|_2^2.
\end{align}

We now estimate the three nonlinear terms in (\ref{con1}) as follows.

Similar to (\ref{ntd2'})--(\ref{ntd22}), we have
\begin{align} \label{con2}
\int_{\Omega} |(\mathbf{U}_{\varepsilon} \cdot \nabla_h \theta') \Theta'_{\varepsilon}| \, dx\,dy\,dz 
\leq C \|\Theta'_{\varepsilon}\|_2^2 \|\nabla_h \theta'\|_3.
\end{align}

Also, mimicking (\ref{ntd33}), we obtain
\begin{align} \label{con3}
\int_{\Omega} |W_{\varepsilon} \partial_z \bar{\theta} \, \Theta'_{\varepsilon}| \, dx\,dy\,dz 
\leq C \|\Theta'_{\varepsilon}\|_2^2 \left( \int_0^{2\pi} \left| \partial_z \bar{\theta} \right|^2 dz \right)^{1/2}.
\end{align}

Moreover, similar to (\ref{inu-8})--(\ref{ntd44}), we derive
\begin{align} \label{con4}
-\int_{\Omega} w_{\varepsilon} \partial_z \bar{\Theta}_{\varepsilon} \, \Theta'_{\varepsilon} \, dx\,dy\,dz 
\leq -2\pi^2 \int_0^{2\pi} \left| \partial_z \bar{\Theta}_{\varepsilon} \right|^2 dz 
+ C \|\theta'\|_2^2 \|\Theta'_{\varepsilon}\|_2^2.
\end{align}

Because of (\ref{con1})-(\ref{con4}), it follows that
\begin{align*}
&\frac{d}{dt} \|\Theta'_{\varepsilon}\|_2^2 + \varepsilon^2 \|\nabla_h \Theta'_{\varepsilon}\|_2^2
+ 4\pi^2 \int_0^{2\pi} \left| \partial_z \bar{\Theta}_{\varepsilon} \right|^2 dz  \notag\\
&\leq \varepsilon^2 \|\nabla_h \theta'\|_2^2 
+ C \|\Theta'_{\varepsilon}\|_2^2 \|\nabla_h \theta'\|_3 
+ C \|\Theta'_{\varepsilon}\|_2^2 \|\bar{\theta}_z\|_2 
+ C \|\theta'\|_2^2 \|\Theta'_{\varepsilon}\|_2^2 \\
&\leq \varepsilon^2 C_2 + C_1 \|\Theta'_{\varepsilon}\|_2^2, \quad \text{for } t \in [0,T],
\end{align*}
where \( C_1 \) and \( C_2 \) depend on \( T \), \( \|\theta'_0\|_{L^6} \), and \( \|\nabla_h \theta'_0\|_3 \).

Therefore,
\begin{align*}
\|\Theta'_{\varepsilon}(t)\|_2^2 
\leq e^{C_1 t} \|\Theta'_{\varepsilon}(0)\|_2^2 + \varepsilon^2 \frac{C_2}{C_1} (e^{C_1 t} - 1), \quad \text{for } t \in [0,T].
\end{align*}

Thus,
\begin{align} \label{con5}
\sup_{t \in [0,T]} \|\Theta'_{\varepsilon}(t)\|_2^2 
\leq C \left( \|\Theta'_{\varepsilon}(0)\|_2^2 + \varepsilon^2 \right),
\end{align}
where \( C \) depends on \( T \), \( \|\theta'_0\|_{L^6} \), and \( \|\nabla_h \theta'_0\|_3 \).

Moreover, using the linear equations (\ref{df-1})--(\ref{df-2}), we deduce from (\ref{con5}) that for every \( t \in [0,T] \),
\begin{align} \label{con7}
\|\Delta_h \mathbf{U}_{\varepsilon}(t)\|_2^2 + \|\Delta_h W_{\varepsilon}(t)\|_2^2 
\leq C \|\Theta'_{\varepsilon}(t)\|_2^2 
\leq C \left( \|\Theta'_{\varepsilon}(0)\|_2^2 + \varepsilon^2 \right).
\end{align}

Finally, we consider \( \bar{\Theta}_{\varepsilon} \). Using a similar estimate as in (\ref{ntd77})--(\ref{ntd88}), we obtain:
\begin{align*}
\int_0^{2\pi} \left| \partial_z \bar{\Theta}_{\varepsilon} \right|^2 dz 
\leq C \|\Theta'_{\varepsilon}\|_2^2 \left( \|\theta'_{\varepsilon}\|_2^2 + \|\theta'\|_2^2 \right) 
\leq C \left( \|\Theta'_{\varepsilon}(0)\|_2^2 + \varepsilon^2 \right),
\end{align*}
for any \( t \in [0,T] \), by virtue of (\ref{con5}), where \( C \) depends on \( T \), \( \|\theta'_0\|_{L^6} \), and \( \|\nabla_h \theta'_0\|_3 \).

This completes the proof of Theorem \ref{thm-conv}.

\section*{Acknowledgements}
The  work E.S.T. was supported in part by the DFG Research Unit FOR 5528 on Geophysical Flows. Moreover, E.S.T. have benefitted from the inspiring environment of the CRC 1114 ``Scaling Cascades in Complex Systems'', Project Number 235221301, Project C09, funded by the Deutsche Forschungsgemeinschaft (DFG).

\end{document}